\documentclass[final]{siamart190516}
\usepackage{amsmath}
\usepackage{amsfonts}
\usepackage{algorithm}
\usepackage{algorithmic}
\usepackage{array}
\usepackage{url}
\usepackage{hyperref}
\usepackage{url}
\usepackage{bm}
\usepackage{graphicx}
\usepackage{verbatim}
\usepackage{placeins}
\usepackage{multicol}
\usepackage{multirow}
\usepackage{capt-of}
\usepackage{booktabs}

\newcommand{\TheTitle}{Adaptive first-order methods with enhanced worst-case rates}
\newcommand{\TheAuthors}{Mihai I. Florea}

\newcommand{\inn}[1]{\in \{ 1, ... , #1 \}}

\title{\TheTitle}

\author{Mihai I. Florea\thanks{Department of Mathematical Engineering (INMA), Catholic University of Louvain (UCL), Belgium; Department of Electronics and Computers, Transilvania University of Bra\c{s}ov, Romania. E-mail: \email{mihai.florea@uclouvain.be}}.
This project has received funding from the European Research Council (ERC) under the European Union's Horizon 2020 research and innovation programme (grant agreement No. 788368).}

\headers{\TheTitle}{\TheAuthors}

\begin{document}

\maketitle

\begin{abstract}
The Optimized Gradient Method (OGM), its strongly convex extension, the Information Theoretical Exact Method (ITEM), as well as the related Triple Momentum Method (TMM) have superior convergence guarantees when compared to the Fast Gradient Method but lack adaptivity and their derivation is incompatible with composite problems. In this work we introduce a slightly modified version of the estimate sequence that can be used to simultaneously derive OGM, ITEM and TMM while adding memory along with the ability to dynamically adjust the convergence guarantees at runtime. Our framework can be extended to the composite setup and we use it to construct an Enhanced Accelerated Composite Gradient Method equipped with fully-adaptive line-search.
\end{abstract}

\section{Introduction}

A very wide array of problems, spanning the areas of mathematics, statistics, computer science, signal processing, optimal control, economics, operations research etc., can be modeled as large-scale composite problems. The objective in these problems can be expressed as the sum between a differentiable function with Lipschitz gradient and a simple proximable regularizer. The extended value regularizer is infinite outside the feasible set and may not be differentiable. The large scale of the problems prevents optimization algorithms used to solve them from using any higher order information, being thus limited to problem constants and first-order oracle function output.

Despite the importance of this problem class, the collection of available tools for algorithm design is much more limited when compared to the better theoretically described subclass of large-scale \emph{smooth unconstrained} problems. One particularly powerful tool that is not readily applicable to composite problems is the Performance Estimation Problem~\cite{ref_002,ref_018}. It has been used to construct the Optimized Gradient Method (OGM) \cite{ref_011} and its strongly-convex extension the Information Theoretical Exact Method (ITEM) \cite{ref_017}. OGM has the best known convergence guarantees on the class of smooth unconstrained problems whereas the guarantees of ITEM on strongly convex problems are unimprovable, even in terms of the proportionality constant. The Triple Momentum Method~\cite{ref_019} has been independently developed using control theoretic arguments~\cite{ref_012}. Interestingly, the introduction of TMM predates that of ITEM and, even though both methods were derived in completely different ways, the two appear to be very closely related. The main difference is that TMM can only be applied to strongly convex objectives whereas ITEM can also operate when strong convexity is absent. But on the restricted class of strongly convex smooth problems, TMM and ITEM have the same asymptotic guarantees. Aside from the PEP and the control theoretical Integral Quadratic Constrains (IQC)~\cite{ref_012,ref_019}, the three methods have been collectively analyzed using Lyapunov (potential) functions~\cite{ref_001,ref_017}. However, neither of the methods feature an adaptive mechanism to increase performance at runtime nor can they deal with constraints. It remains unclear whether the PEP, IQC or the potential function approaches can be used to adapt these methods to the class of composite problems while retaining any adaptive mechanism other than restarting schemes.

In \cite{ref_007}, we have devised primal-dual bounds on smooth objectives that, combined with a novel primal-dual estimate sequence framework, can be used to generate a first-order method having the same asymptotic convergence guarantees as OGM while also featuring a convergence guarantee adjustment procedure. Preliminary simulation results suggest that, at least for difficult problems, the dynamic convergence guarantee adjustment is an effective substitute for line-search.

The primal-dual bounds in \cite{ref_007} were designed to solve the oracle output interpolation problem. However, they introduced a slack term in the analysis that resulted in convergence guarantees that, despite being asymptotically optimal even up to the proportionality constant, are slightly lower than those of OGM. The reason behind this discrepancy is the lack of any slack term in the convergence analysis of OGM. Based on this observation, in this work we introduce a form for the primal-dual bounds that is specifically designed to maximize the worst-case guarantees. Our analysis is readily applicable to strongly convex objectives. We enhance the primal-dual estimate functions to better deal with strong convexity and we equip them with our improved bounds.

We then use our machinery to derive an Optimized Gradient Method with Memory (OGMM) applicable to strongly convex objectives that encompasses OGM, TMM and ITEM as particular cases. Unlike in \cite{ref_007}, our analysis no longer contains slack terms. The estimate function powers a dynamic convergence guarantee adjustment procedure, previously used in \cite{ref_004,ref_006,ref_007}, that can also be combined with a subset of the oracle history to further improve performance.

The enhanced estimate functions are not limited to primal-dual bounds and can accept global lower bounds that apply to composite optimization problems, in the process taking advantage of the Quadratic Functional Growth property~\cite{ref_013}, a consequence of strong convexity. Our new estimate function framework combined with the bounds derived in \cite{ref_009,ref_010} \emph{directly} yields an Enhanced Accelerated Composite Gradient Method (EACGM) that features fully adaptive line-search and currently has the highest known guarantees in iterate space for an adaptive scheme applicable to the entire composite problem class.

Preliminary simulation results on synthetic smooth problems confirm our theoretical findings and show that in iterate space the Optimized Gradient Method with Memory outperforms both ITEM and TMM and, interestingly, is competitive with the state-of-the-art adaptive Accelerated Composite Gradient Method (ACGM)~\cite{ref_009,ref_008}. The Enhanced Accelerated Composite Gradient Method also surpasses ACGM on the synthetic composite problem benchmark.

\section{Optimal methods for smooth unconstrained minimization}

In this first part, we focus exclusively on smooth unconstrained problems. We build the relevant mathematical machinery and use it to derive a version of the Optimized Gradient Method with Memory with improved worst-case guarantees that can effectively deal with known strong convexity.

\subsection{Problem description}

We consider the following smooth unconstrained optimization problem:
\begin{equation} \label{label_001}
\min_{x \in \mathbb{E}} f(x).
\end{equation}
The function $f:\mathbb{E} \rightarrow \mathbb{R}$ is convex and differentiable over the entire $\mathbb{E}$. The gradient of $f$ is Lipschitz with constant $L_f > 0$ ($f$ is thus $L_f$-smooth) and $f$ has a strong convexity parameter $\mu \in [0, L_f)$. We define the Euclidean norm as $\| x \| \overset{\operatorname{def}}{=} \sqrt{ \langle B x, x \rangle }$, $x \in \mathbb{E}$ with $B \succ 0$. The dual norm is $\| g \|_* \overset{\operatorname{def}}{=} \sqrt{ \langle g, B^{-1} g \rangle }$, $g \in \mathbb{E}^*$.
The Lipschitz gradient and strong convexity can equivalently written as the following combined property for all $z_1, z_2 \in \mathbb{E}$:
\begin{equation} \label{label_002}
\frac{\mu}{2} \| z_1 - z_2 \|^2 \leq f(z_1) - f(z_2) - \langle f'(z_2), z_1 - z_2 \rangle \leq \frac{L_f}{2} \| z_1 - z_2 \|^2 .
\end{equation}
To simplify the derivations we introduce the inverse condition number $q \in [0, 1)$ and the ratio $r \geq 1$, respectively given by
\begin{equation}
q \overset{\operatorname{def}}{=} \frac{\mu}{L_f}, \quad r \overset{\operatorname{def}}{=} \frac{L_f}{L_f-\mu} = \frac{1}{1 - q}.
\end{equation}
A first-order scheme calls the oracle at the points $y_k$, $k \geq 1$, to obtain the pair $(f(y_k), g_k)$, where $g_k \overset{\operatorname{def}}{=} f'(y_k)$, $k \geq 1$. Only the starting point $y_1 = x_0$ is known to the algorithm, the rest of the points being selected using the output of the oracle at the previous points. The Lipschitz gradient property along with Majorization-Minimization framework (see \cite{ref_007} and the references therein) at the points where the oracle is called produces the additional sequence
\begin{equation} \label{label_003}
x_k \overset{\operatorname{def}}{=} T_{L_f}(y_k) = y_k - \frac{1}{L_f} B^{-1} g_k, \quad k \geq 1,
\end{equation}
where $T_{L}(z) \overset{\operatorname{def}}{=} z - \frac{1}{L} B^{-1} f'(z)$ for all $z \in \mathbb{E}$, $L > 0$.
Note that $x_0$ designates the starting point and does not adhere to \eqref{label_003}.

We consider the problem $\min_{x \in \mathbb{E}}f(x)$ and assume it has a non-empty optimal set $X^*$. We select one element $x^*$ from this set and consider it fixed, unless specified otherwise.

\subsection{Primal-dual bounds}
We define the counterpart of the objective $f$ with the strong convexity removed as
\begin{equation} \label{label_004}
f_{(\mu)}(y) \overset{\operatorname{def}}{=} f(y) - \frac{\mu}{2}\| y \|^2, \quad y \in \mathbb{E} .
\end{equation}
Obviously, $f_{(\mu)}$ is convex and has a $(L_f-\mu)$-Lipschitz gradient.
It is a well-known result (see, e.g.~\cite{ref_014}) that \eqref{label_002} can be equivalently expressed as a single inequality satisfied for all $z_1, z_2 \in \mathbb{E}$ as
\begin{equation} \label{label_005}
f_{(\mu)}(z_1) \geq f(z_2) + \left\langle f_{(\mu)}'(z_2), z_1 - z_2 \right\rangle + \frac{1}{2 (L_f - \mu)} \| f_{(\mu)}'(z_1) - f_{(\mu)}'(z_2) \|_*^2.
\end{equation}

Expanding \eqref{label_004} and noting that $f_{(\mu)}'(y) = f'(y) - \mu B y$ for all $y \in \mathbb{E}$, we equivalently obtain for all $z_1, z_2 \in \mathbb{E}$ a bound in a canonical form that completely describes the problem~\cite{ref_018}:
\begin{equation} \label{label_006}
\begin{gathered}
f(z_1) \geq f(z_2) + \langle f'(z_2), z_1 - z_2 \rangle \\
+ r \left( \frac{\mu}{2} \| z_1 - z_2 \|^2 + \frac{1}{2 L_f} \| f'(z_1) - f'(z_2) \|_*^2 - \frac{\mu}{L} \langle f'(z_1) - f'(z_2), z_1 - z_2 \rangle \right).
\end{gathered}
\end{equation}
We refactor the bound in \eqref{label_006} to obtain a form that is key to obtaining an optimal method for this problem class
\begin{equation} \label{label_007}
\begin{gathered}
f(z_1) \geq f(z_2) + \langle f'(z_2), T_{L_f}(z_1) - z_2 \rangle \\ + \frac{\mu r}{2} \| T_{L_f}(z_1) - T_{L_f}(z_2) \|^2 + \frac{1}{2 L_f} \| f'(z_1) \|_*^2 + \frac{1}{2 L_f} \| f'(z_2) \|_*^2.
\end{gathered}
\end{equation}

We consider the following primal-dual bound centered around $y_k$ defined for all $x \in \mathbb{E}$, $g \in \mathbb{E}^*$ and $k \geq 1$ as
\begin{equation} \label{label_008}
w_k(x, g) \overset{\operatorname{def}}{=} f(y_k) + \frac{1}{2 L_f} \| g_k \|_*^2 + \langle g_k, x - y_k \rangle + \frac{\mu r}{2} \|x - x_k \|^2 + \frac{1}{2 L_f} \| g \|_*^2.
\end{equation}
From \eqref{label_007} taken with $z_1 = y$ and $z_2 = y_k$ we have that
\begin{equation} \label{label_009}
f(y) \geq w_k(T_{L_f}(y), f'(y)), \quad y \in \mathbb{E}, \quad k \geq 1.
\end{equation}

The novelty of our new bounds in \eqref{label_008} over the the primal-dual bounds introduced in~\cite{ref_007} lies in the use of the combined variable $x$. By performing the substitution
\begin{equation} \label{label_010}
x = y - \frac{1}{L_f} B^{-1} g,
\end{equation}
we quickly recover from \eqref{label_008} the primal-dual bound form in \cite{ref_007}. However, as we will show in the sequel, the variable change in \eqref{label_010} is essential in eliminating the slack terms in the analysis to obtain the fast possible rate.

Another component necessary in obtaining the absolute optimal rate is the auxiliary bound defined as
\begin{equation} \label{label_011}
\hat{w}_k(x) \overset{\operatorname{def}}{=} \frac{\mu r}{2} \| x - x_k \|^2 + \frac{1}{2 L_f} \| g_k \|_*^2, \quad x \in \mathbb{E}, \quad k \geq 1.
\end{equation}
The smooth unconstrained nature of the problem ensures that $f'(x^*) = 0$ which yields the identity $T_{L_f}(x^*) = x^*$. Further taking \eqref{label_007}, this time with $z_1 = y_k$ and $z_2 = x^*$, produces
\begin{equation} \label{label_012}
f(y_k) \geq f^* + \hat{w}_k(x^*), \quad k \geq 1.
\end{equation}

\subsection{The estimate functions}

Our modified bounds imply a corresponding modification of the primal-dual estimate functions introduced in \cite{ref_007}. After performing the variable change in \eqref{label_010}, we consider that
the definitions and results in this section apply to all $x \in \mathbb{E}$ and $g \in \mathbb{E}^*$.

Recall that the primal-dual estimate functions in \cite{ref_007} employ the convergence guarantees in the aggregation of simple primal-dual bounds. This aspect is retained in our current construction. Let the convergence guarantees at the start of every iteration $k \geq 1$ be denoted by $A_k$ with $A_1 \geq 0$ and $A_{k + 1} > A_{k}$, $k \geq 1$. The positive difference terms are given by $a_{k + 1} = A_{k + 1} - A_{k} > 0$ for all $k \geq 1$. The simplest means of aggregating the lower bounds in \eqref{label_008} is by weighted averaging into the compound $W_k$, namely
\begin{equation} \label{label_013}
(A_k - A_1) W_k(x, g) = \displaystyle \sum_{i = 2}^{k} a_i w_i(x, g), \quad k \geq 2.
\end{equation}
Note that our aggregation technique in \eqref{label_014} implies that $w_1$ is not included in the weighted average. In practice, this translates to starting our algorithm at the point $x_1$, obtained from $y_1 = x_0$ by taking the gradient step in \eqref{label_003}. This convention effectively addresses the case when $A_1 > 0$, as we will outline in the sequel.

Using auxiliary bounds $\hat{w}_k$ given by \eqref{label_011} alongside our aggregate $W_k$ in \eqref{label_013}, we propose the estimate functions taking for all $k \geq 1$ the form
\begin{equation} \label{label_014}
\psi_k(x, g) = (A_k - A_1) W_k(x, g) + A_k \hat{w}_k(x) + A_1 (f(y_1) - \hat{w}_1(x)) + \frac{\gamma_1}{2} \| x - v_1 \|^2.
\end{equation}
Here, the regularizer is quadratic term defined by the curvature $\gamma_1 > 0$, which is an algorithm constant, and the point $v_1$ that can be arbitrary in $\mathbb{E}$. The estimate function optima are given by
\begin{equation} \label{label_015}
(v_k, s_k) = \arg\min_{x, g} \psi_k(x, g), \quad k \geq 1.
\end{equation}
We have $\psi_1(x, g) = A_1 f(y_1) + \frac{\gamma_1}{2} \| x - v_1 \|^2$ with $\psi_1^* = A_1 f(y_1)$ whereas $W_1$ and $s_1$ are left undefined. Arguably the best choice for $v_1$ is $x_1$ because, right after the computation of $g_1$, $x_1$ is the closest point to the optimal set that can be constructed without relying on additional oracle calls. The refactored structure of the bounds in \eqref{label_008} and \eqref{label_011} ensure that $s_k = 0$, $k \geq 2$, and that the estimate functions take the canonical form
\begin{equation} \label{label_016}
\psi_k(x, g) =\psi_k^* + \frac{\gamma_k}{2} \| x - v_k \|^2 + \frac{A_k - A_1}{2 L_f} \| g \|_*^2, \quad k \geq 1,
\end{equation}
where
\begin{equation} \label{label_017}
\gamma_{k + 1} = \gamma_k + 2 \mu r a_{k + 1}, \quad k \geq 1.
\end{equation}
Iterating for all $k \geq 1$ we obtain $\gamma_{k} = \gamma_1 + 2 \mu r (A_k - A_1)$.

The estimate functions in \eqref{label_014} thus differ from the primal-dual estimate functions previously introduced in \cite{ref_007} by the addition of the term $\hat{w}_k$, counterbalanced so that the initial estimate function remains unchanged. This increases the growth rate of the curvature in \eqref{label_017}. Moreover, the variable change in \eqref{label_010} effectively eliminates the dual optimum $s_k$ in \eqref{label_015} from the canonical form in \eqref{label_016}.

In the sequel we show that we need only to maintain the Estimate Sequence Property (ESP), stated as
\begin{equation} \label{label_018}
A_k f(y_k) \leq \psi_k^*, \quad k \geq 1,
\end{equation}
to produce a valid algorithm with $A_k$ as the convergence guarantees.
The ESP along with our assumption leads to the following chain of inequalities:
\begin{equation} \label{label_019}
\begin{gathered}
A_k f(y_k) + \frac{\gamma_k}{2} \| v_k - x^* \|^2 \overset{\eqref{label_016}}{\leq} \psi_k(x^*, 0) \overset{\eqref{label_014}}{\leq} (A_k - A_1) f^* + A_k \hat{w}_k(x^*) \\ + A_1 ( f(y_1) - \hat{w}_1(x^*) ) + \frac{\gamma_1}{2} \| y_1 - x^* \|^2 \overset{\eqref{label_012}}{\leq} A_k f(y_k) + \mathcal{D}_1, \quad k \geq 1,
\end{gathered}
\end{equation}
where $\mathcal{D}_1 \overset{\operatorname{def}}{=} A_1 (f(y_1) - f^* - \hat{w}_1(x^*) ) + \frac{\gamma_1}{2} \| v_1 - x^* \|^2$. Applying \eqref{label_012} with $k = 1$ confirms that $\mathcal{D}_1 \geq 0$.

We can obtain convergence guarantees in two ways. First, we only apply the first two inequalities in \eqref{label_019} to obtain
\begin{equation} \label{label_018_1}
f(y_k) - \hat{w}_k(x^*) - f^* \leq \frac{\mathcal{D}_1}{A_k}, \quad k \geq 1.
\end{equation}
When $\mu = 0$, \eqref{label_018_1} implies that
\begin{equation} \label{label_021}
f(x_k) \leq f(y_k) - \frac{1}{2 L_f} \| g_k \|_*^2 \leq \frac{\mathcal{D}_1}{A_k}, \quad k \geq 1.
\end{equation}
The second means of obtaining convergence guarantees becomes available only when $\mu > 0$. Canceling the $A_k f(y_k)$ term on both sides of the whole \eqref{label_019} yields
\begin{equation} \label{label_022}
f(v_k) - f^* \leq \frac{L_f}{2} \| v_k - x^* \|^2 \leq \frac{L_f \mathcal{D}_1}{\gamma_k} = \frac{L_f \mathcal{D}_1}{\gamma_1 + 2 \mu r (A_k - A_1)}, \quad k \geq 1.
\end{equation}

\subsection{Towards an algorithm}

The aggregate weight and estimate function definitions in \eqref{label_013} and \eqref{label_014} result in an estimate function update that notably departs from the first-order design pattern previously proposed in \cite{ref_009,ref_010,ref_003}. Combining \eqref{label_013} and \eqref{label_014} results in the new update
\begin{equation} \label{label_023}
\psi_{k + 1}(x, g) - A_{k + 1} \hat{w}_{k + 1}(x) = \psi_k(x, g) - A_k \hat{w}_k(x) + a_{k + 1} w_{k + 1}(x, g), \quad k \geq 1.
\end{equation}
Expanding all terms in \eqref{label_023} using the canonical form in \eqref{label_016} and our bound choices in \eqref{label_008} and \eqref{label_011}, respectively, we obtain for all $k \geq 1$ that
\begin{equation} \label{label_024}
\begin{gathered}
\psi_{k + 1}^* + \frac{\gamma_{k + 1}}{2} \| x - v_{k + 1} \|^2 + \frac{A_{k + 1} - A_1}{2 L_f} \| g \|_*^2 - \frac{\mu r A_{k + 1}}{2} \| x - x_{k + 1} \|^2 - \frac{A_{k + 1}}{2 L_f} \| g_{k + 1} \|_*^2 \\
= \psi_k^* + \frac{\gamma_k}{2} \| x - v_k \|^2 + \frac{A_k - A_1}{2 L_f} \| g \|_*^2 - \frac{\mu r A_k}{2} \| x - x_k \|^2 - \frac{A_k}{2 L_f} \| g_k \|_*^2 \\
+ a_{k + 1} \left( f(y_{k + 1}) + \frac{1}{2 L_f} \| g_{k + 1} \|_*^2 + \langle g_{k + 1}, x - y_{k + 1} \rangle + \frac{\mu r}{2} \| x - x_{k + 1} \|^2 + \frac{1}{2 L_f} \| g \|_*^2 \right).
\end{gathered}
\end{equation}
We define the estimate sequence gaps in \eqref{label_018} as $\Gamma_k \overset{\operatorname{def}}{=} \psi_k^* - A_k f(y_k)$, $k \geq 1$. Rearranging terms in \eqref{label_024} gives
\begin{equation} \label{label_025}
\begin{gathered}
\Gamma_{k + 1} - \Gamma_k = \frac{\gamma_{k}}{2} \| x - v_{k} \|^2 - \frac{\gamma_{k + 1}}{2} \| x - v_{k + 1} \|^2 + \frac{\mu r (A_{k + 1} + a_{k + 1})}{2} \| x - x_{k + 1} \|^2 \\ - \frac{\mu r A_k}{2} \| x - x_k \|^2
+ \frac{A_{k + 1} + a_{k + 1}}{2 L_f} \| g_{k + 1} \|_*^2 - \frac{A_k}{2 L_f} \| g_k \|_*^2 \\ + a_{k + 1} \langle g_{k + 1}, x - y_{k + 1} \rangle + A_k(f(y_k) - f(y_{k + 1})), \quad k \geq 1.
\end{gathered}
\end{equation}
Differentiating \eqref{label_024} with respect to $x$ and then setting $x = y_{k + 1}$ gives
\begin{equation} \label{label_026}
\begin{gathered}
\gamma_{k + 1} (v_{k + 1} - y_{k + 1}) = \gamma_k (v_k - y_{k + 1}) + \mu r (A_{k + 1} + a_{k + 1}) (x_{k + 1} - y_{k + 1}) \\ - \mu r A_k (x_k - y_{k + 1}) - a_{k + 1} g_{k + 1}
\overset{\eqref{label_003}}{=} \gamma_k (v_k - y_{k + 1}) - \mu r A_k (x_k - y_{k + 1}) \\ - (a_{k + 1} + q r (A_{k + 1} + a_{k + 1})) B^{-1} g_{k + 1}, \quad k \geq 1.
\end{gathered}
\end{equation}
To simplify our exposition, we define the coefficient of $B^{-1} g_{k + 1}$ as $\bar{a}_{k + 1}$ and refactor it using $1 + r q = r$ as
\begin{equation} \label{label_027}
\bar{a}_{k + 1} \overset{\operatorname{def}}{=} a_{k + 1} + q r (A_k + 2 a_{k + 1}) = r (a_{k + 1} + q A_{k + 1}), \quad k \geq 1.
\end{equation}
The estimate function optima update in \eqref{label_026} can thus be equivalently expressed as
\begin{align}
\gamma_{k + 1} v_{k + 1} &= \gamma_k v_k + \mu r (A_{k + 1} + a_{k + 1}) x_{k + 1} - \mu r A_k x_k - a_{k + 1} B^{-1} g_{k + 1} \label{label_028} \\
& = \gamma_k v_k + \mu r (A_{k + 1} + a_{k + 1}) y_{k + 1} - \mu r A_k x_k - \bar{a}_{k + 1} B^{-1} g_{k + 1} \label{label_029}
\end{align}

Applying \eqref{label_007} with $z_1 = y_k$ and $z_2 = y_{k + 1}$ we obtain
\begin{equation} \label{label_030}
\begin{gathered}
A_k (f(y_k) - f(y_{k + 1})) \geq A_k \langle g_{k + 1}, x_k - y_{k + 1} \rangle \\ + \frac{\mu r A_k}{2} \| x_k - x_{k + 1} \|^2 + \frac{A_k}{2 L_f} \| g_k \|_*^2 + \frac{A_k}{2 L_f} \| g_{k + 1} \|_*^2, \quad k \geq 1.
\end{gathered}
\end{equation}
Further using \eqref{label_003}, we also have for all $k \geq 1$ the identity
\begin{equation} \label{label_031}
\frac{\mu r A_k}{2} \left( \|x_k - x_{k + 1}\|^2 - \| y_k - x_{k + 1} \|^2 \right) = \frac{q r A_k}{2} \langle g_{k + 1}, x_k - y_{k + 1} \rangle + \frac{q r A_k}{2 L_f} \| g_{k + 1} \|_*^2.
\end{equation}
Setting $x = y_{k + 1}$ in \eqref{label_025}, combining with \eqref{label_030} and grouping terms using $1 + qr = r$ and \eqref{label_031} reveals the remarkably simple inequality
\begin{equation} \label{label_032}
\begin{gathered}
\Gamma_{k + 1} - \Gamma_k \geq \frac{\gamma_{k}}{2} \| v_{k} - y_{k + 1} \|^2 - \frac{\gamma_{k + 1}}{2} \| v_{k + 1} - y_{k + 1} \|^2 \\ + r A_k \langle g_{k + 1}, x_k - y_{k + 1} \rangle + \frac{r A_{k + 1}}{L_f} \| g_{k + 1} \|_*^2, \quad k \geq 1.
\end{gathered}
\end{equation}
Now we can formulate the main result of this section.
\begin{theorem} \label{label_033}
When the estimate functions are updated using \eqref{label_024} we have
\begin{equation}
\begin{gathered}
\Gamma_{k + 1} - \Gamma_k \geq
\left( \frac{r A_{k + 1}}{L_f} - \frac{\bar{a}_{k + 1}^2}{2 \gamma_{k + 1}} \right)
\left( \| g_{k + 1} \|_*^2 - \frac{\mu^2 \gamma_k \gamma_{k + 1}}{\bar{\gamma}^2_{k + 1}} \| v_k - y_{k + 1} \|^2 \right) \\
+ \left\langle \frac{\bar{\gamma}_{k + 1}}{\gamma_{k + 1}} g_{k + 1} + \mu B \left( \frac{\gamma_k}{\bar{\gamma}_{k + 1}} (v_k - y_{k + 1}) - \frac{\mu}{2 \gamma_{k + 1}} Y_{k + 1} \right), Y_{k + 1} \right\rangle,
\end{gathered}
\end{equation}
for all $k \geq 0$ where
\begin{align}
\bar{\gamma}_{k + 1} &\overset{\operatorname{def}}{=} \gamma_{k + 1} - \mu \bar{a}_{k + 1} ,\label{label_034} \\
Y_{k + 1} &\overset{\operatorname{def}}{=} r A_k (x_k - y_{k + 1}) + \frac{\bar{a}_{k + 1} \gamma_k}{\bar{\gamma}_{k + 1}}(v_k - y_{k + 1}) \label{label_035} .
\end{align}
\end{theorem}
\begin{proof}
Throughout this proof we consider all $k \geq 1$. From \eqref{label_035} we have
\begin{equation} \label{label_036}
-\mu r A_k (x_k - y_{k + 1}) = \frac{\mu \bar{a}_{k + 1} \gamma_k}{\bar{\gamma}_{k + 1}}(v_k - y_{k + 1}) - \mu Y_{k + 1}.
\end{equation}
Substituting \eqref{label_036} into \eqref{label_026} and using \eqref{label_034} we obtain
\begin{equation} \label{label_037}
\gamma_{k + 1}(v_{k + 1} - y_{k + 1}) = \frac{\gamma_k \gamma_{k + 1}}{\bar{\gamma}_{k + 1}} (v_k - y_{k + 1}) - \mu Y_{k + 1} - \bar{a}_{k + 1} B^{-1} g_{k + 1} .
\end{equation}
We thus have
\begin{equation} \label{label_038}
\begin{gathered}
\frac{\gamma_{k}}{2} \| v_{k} - y_{k + 1} \|^2 - \frac{\gamma_{k_+ 1}}{2} \| v_{k + 1} - y_{k + 1} \|^2 \\= \frac{\gamma_{k}}{2} \| v_{k} - y_{k + 1} \|^2 - \frac{1}{2 \gamma_{k_+ 1}} \left\| \frac{\gamma_k \gamma_{k + 1}}{\bar{\gamma}_{k + 1}}(v_k - y_{k + 1}) - \mu Y_{k + 1} - \bar{a}_{k + 1} B^{-1} g_{k + 1} \right\|^2 \\
= \frac{\gamma_k}{2 \bar{\gamma}_{k + 1}^2} \left( \bar{\gamma}_{k + 1}^2 - \gamma_k \gamma_{k + 1} \right) \| v_k - y_{k + 1} \|^2 - \frac{\mu^2}{2 \gamma_{k + 1}} \| Y_{k + 1} \|^2 - \frac{\bar{a}_{k + 1}^2}{2 \gamma_{k + 1}} \| g_{k + 1} \|_*^2 \\+ \frac{\bar{a}_{k + 1} \gamma_k}{\bar{\gamma}_{k + 1}} \langle g_{k + 1}, v_k - y_{k + 1} \rangle + \left\langle \frac{\mu \gamma_k }{\bar{\gamma}_{k + 1}} B (v_k - y_{k + 1}) - \frac{\mu \bar{a}_{k + 1}}{\gamma_{k + 1}} g_{k + 1} , Y_{k + 1} \right\rangle .
\end{gathered}
\end{equation}
Substituting \eqref{label_038} into \eqref{label_032} and regrouping terms using \eqref{label_035} gives
\begin{equation} \label{label_039}
\begin{gathered}
\Gamma_{k + 1} - \Gamma_k \geq
\left( \frac{r A_{k + 1}}{L_f} - \frac{\bar{a}_{k + 1}^2}{2 \gamma_{k + 1}} \right)
\| g_{k + 1} \|_*^2 + \frac{\gamma_k}{2 \bar{\gamma}_{k + 1}^2} \left( \bar{\gamma}_{k + 1}^2 - \gamma_k \gamma_{k + 1} \right) \| v_k - y_{k + 1} \|^2 \\
+ \left\langle \left(1 - \frac{\mu \bar{a}_{k + 1}}{\gamma_{k + 1}} \right) g_{k + 1} + \mu B \left( \frac{\gamma_k}{\bar{\gamma}_{k + 1}} (v_k - y_{k + 1}) - \frac{\mu}{2 \gamma_{k + 1}} Y_{k + 1} \right), Y_{k + 1} \right\rangle.
\end{gathered}
\end{equation}
On the other hand, we have
\begin{equation} \label{label_040}
\begin{gathered}
\bar{\gamma}_{k + 1}^2 - \gamma_k \gamma_{k + 1} = (\gamma_{k + 1} - \mu \bar{a}_{k + 1})^2 - (\gamma_{k + 1} - 2 \mu r a_{k + 1}) \gamma_{k + 1} \\= \mu (\mu \bar{a}_{k + 1}^2 - 2 \gamma_{k + 1} (\bar{a}_{k + 1} - r a_{k + 1}) \overset{\eqref{label_027}}{=} 2 \mu^2 \gamma_{k + 1}\left(\frac{\bar{a}_{k + 1}^2}{2 \gamma_{k + 1}} - \frac{r A_{k + 1}}{L_f} \right)
\end{gathered}
\end{equation}
Putting together \eqref{label_039} and \eqref{label_040} applying \eqref{label_034} completes the proof.
\end{proof}
Theorem~\ref{label_033} provides us with all the information we need to construct an optimal method. We have seen that if we can maintain $\Gamma_k \geq 0$ at every iteration, the convergence of the scheme is guaranteed. We obviously have $\Gamma_1 = 0$ and a sufficient condition for achieving convergence is thus $\Gamma_{k + 1} \geq \Gamma_k$, $k \geq 1$. The simplest way of achieving non-decreasing gaps via Theorem~\ref{label_033} is to enforce
\begin{align}
L_f \bar{a}^2_{k + 1} &= 2 r A_{k + 1} \gamma_{k + 1}, \label{label_041} \\
Y_{k + 1} &= r A_k (x_k - y_{k + 1}) + \frac{\bar{a}_{k + 1} \gamma_k}{\bar{\gamma}_{k + 1}} (v_k - y_{k + 1}) = 0. \label{label_044_raw}
\end{align}
Note that unlike the Fast Gradient Method~\cite{ref_014} and extensions such as ACGM~\cite{ref_009,ref_010}, we must have equality in \eqref{label_041} to account a priori for all possible algorithmic conditions. Equality in \eqref{label_041} leads to the following second degree equation
\begin{equation}
(L_f - \mu) a_{k + 1}^2 - 2 \mu a_{k + 1} - A_k \left( 2 \gamma_k - \frac{\mu^2 A_k}{L_f - \mu} \right) = 0, \quad k \geq 1,
\end{equation}
allowing a single positive solution
\begin{equation} \label{label_043}
a_{k + 1} = \frac{1}{L_f - \mu} \left( \gamma_k + \mu A_k + \sqrt{\gamma_k (\gamma_k + 2 L_f A_k)} \right), \quad k \geq 1.
\end{equation}

By refactoring \eqref{label_044_raw} we also obtain an update rule for points where the oracle is called
\begin{equation} \label{label_044}
y_{k + 1} = (r A_k \bar{\gamma}_{k + 1} + \bar{a}_{k + 1} \gamma_k)^{-1} \left( r A_k \bar{\gamma}_{k + 1} x_k + \bar{a}_{k + 1} \gamma_k v_k \right), \quad k \geq 1.
\end{equation}
Applying \eqref{label_044_raw} in \eqref{label_026} we have that
\begin{equation} \label{label_045}
v_{k + 1} = \frac{\gamma_k}{\bar{\gamma}_{k + 1}} v_k + \left(1 - \frac{\gamma_k}{\bar{\gamma}_{k + 1}} \right) y_{k + 1} - \frac{\bar{a}_{k + 1}}{\gamma_{k + 1}} B^{-1} g_{k + 1}, \quad k \geq 1.
\end{equation}
From \eqref{label_041} we also obtain via \eqref{label_040} that $\bar{\gamma}_{k + 1}^2 = \gamma_k \gamma_{k + 1}$, which enables us to further simplify \eqref{label_045} as
\begin{equation} \label{label_046}
v_{k + 1} = \frac{1}{\gamma_{k + 1}} (\bar{\gamma}_{k + 1} v_k - \bar{a}_{k + 1} ( B^{-1} g_{k + 1} - \mu y_{k + 1} ) ), \quad k \geq 1.
\end{equation}

Putting together the weight update in \eqref{label_043}, the estimate function curvature update in \eqref{label_017}, the auxiliary quantity updates in \eqref{label_027} and \eqref{label_034}, respectively, the auxiliary point update in \eqref{label_044} as well as the estimate function optima update in \eqref{label_046} we obtain Algorithm~\ref{label_047}.

\begin{algorithm}[h!]
\caption{A generalized Optimized Gradient Method}
\label{label_047}
\begin{algorithmic}[1]
\STATE \textbf{Input:} $B \succ 0$, $x_0 \in \mathbb{E}$, $\mu \geq 0$, $L_f > 0$, $A_1 \geq 0$, $\gamma_1 > 0$, $T \inn{\infty}$
\STATE $y_1 = x_0$
\STATE $g_{1} = f'(y_{1})$\\[1mm]
\STATE $x_1 = y_1 - \frac{1}{L_f} B^{-1} g_1$
\STATE Choose $v_1$ from among $x_1$ (recommended), $x_0$, or an additional input point
\FOR{$k = 1,\ldots{},T$}
\STATE $a_{k + 1} = \frac{1}{L_f - \mu} \left( \gamma_k + \mu A_k + \sqrt{\gamma_k (\gamma_k + 2 L_f A_k)} \right)$ \label{label_048} \\[1mm]
\STATE $A_{k + 1} = A_k + a_{k + 1}$ \\
\STATE $\gamma_{k + 1} = \gamma_k + 2 \mu r a_{k + 1}$
\STATE $\bar{a}_{k + 1} = r (a_{k + 1} + q A_{k + 1})$
\STATE $\bar{\gamma}_{k + 1} = \gamma_{k + 1} - \mu \bar{a}_{k + 1}$
\STATE $y_{k + 1} = (r A_k \bar{\gamma}_{k + 1} + \bar{a}_{k + 1} \gamma_k)^{-1} \left( r A_k \bar{\gamma}_{k + 1} x_k + \bar{a}_{k + 1} \gamma_k v_k \right)$\\[1mm]
\STATE $g_{k + 1} = f'(y_{k + 1})$
\STATE $x_{k + 1} = y_{k + 1} - \frac{1}{L_f} B^{-1} g_{k + 1}$
\STATE $v_{k + 1} = \frac{1}{\gamma_{k + 1}} (\bar{\gamma}_{k + 1} v_k - \bar{a}_{k + 1} ( B^{-1} g_{k + 1} - \mu y_{k + 1} ) )$
\ENDFOR
\end{algorithmic}
\end{algorithm}

\subsection{Convergence analysis}

Our method maintains the ESP in \eqref{label_018} and therefore the chain of inequalities in \eqref{label_019} holds. We again distinguish two cases.

When $\mu = 0$, we have $\bar{\gamma}_k = \gamma_k = \gamma_1$ and $\bar{a}_k = a_k$ for all $k \geq 2$. It can be proven by induction that $A_k \geq \frac{\gamma_1}{2 L_f} k^2$ for all $k \geq 2$ (see e.g. \cite{ref_001}). From \eqref{label_021} we have
\begin{equation} \label{label_049}
f(x_k) \leq f(y_k) - \frac{1}{2 L_f} \| g_k \|_*^2 \leq \frac{\mathcal{D}_1}{A_k} \leq \frac{L_f}{k^2} \bar{\mathcal{D}}_1, \quad k \geq 2,
\end{equation}
where
\begin{equation}
\bar{\mathcal{D}}_1 \overset{\operatorname{def}}{=} \frac{2 \mathcal{D}_1}{\gamma_1} = \frac{2 A_1}{\gamma_1} \left(f(x_0) - \frac{1}{2 L_f} \|f'(x_0)\|^2 - f^* \right) + \| v_1 - x^* \|^2.
\end{equation}
When $\mu > 0$ we need to impose the additional condition $\gamma_1 \geq 2 \mu r A_1$. Under these extra assumptions, we have from \eqref{label_041} that
\begin{equation} \label{label_050}
L_f \bar{a}_{k + 1}^2 = 2 r A_{k + 1} \gamma_{k + 1} \geq 2 r A_{k + 1} (2 \mu r A_{k + 1}) = 4 \mu r^2 A_{k + 1}^2, \quad k \geq 1.
\end{equation}
Using \eqref{label_027} and $A_{k + 1} > 0$ in \eqref{label_050} we obtain $\left(\frac{a_{k + 1}}{A_{k + 1}} + q\right)^2 \geq 4 q$, which ultimately gives
\begin{equation} \label{label_051}
A_{k + 1} \geq (1 - \sqrt{q})^{-2} A_k, \quad k \geq 1.
\end{equation}
We also have that $A_2 \geq a_2 \geq \frac{2 \gamma_1}{L_f - \mu}$ (equality is attained when $A_1 = 0$) and iterating \eqref{label_051} yields
\begin{equation} \label{label_052}
A_{k} \geq (1 - \sqrt{q})^{4 - 2 k} \frac{2 \gamma_1}{L_f - \mu}, \quad k \geq 2.
\end{equation}
The chain in \eqref{label_019} together with the rate in \eqref{label_052} provides a worst-case rate for the distance to the optimum of the iterates $v_k$ as
\begin{equation} \label{label_053}
\| v_k - x^* \|^2 \leq \frac{2 \mathcal{D}_1}{\gamma_k} \leq \frac{\mathcal{D}_1}{\mu r A_k} \leq \left(1 - \sqrt{q} \right)^{2 k - 4} \frac{(1 - q)^2}{4 q} \bar{\mathcal{D}}_1, \quad k \geq 2.
\end{equation}
The Lipschitz gradient property of $f$ immediately produces an optimal worst-case rate for function values as well, given by
\begin{equation} \label{label_054}
f(v_k) - f^* \leq \frac{L_f}{2} \| v_k - x^* \|^2 \leq \left(1 - \sqrt{q} \right)^{2 k - 4} \frac{(1 - q)^2 L_f}{8 q} \bar{\mathcal{D}}_1, \quad k \geq 2.
\end{equation}
Note that our method does not adaptively estimate the Lipschitz constant in \eqref{label_054} and the performance of our method when measured using the function residual is not competitive in practice with methods optimized for this criterion such as ACGM~\cite{ref_009}.

\subsection{Relationship with existing methods}

Our generalized Optimized Gradient Method is an umbrella method, encompassing the original OGM, ITEM and TMM. Specifically, the update rules in \eqref{label_043}, \eqref{label_044} and \eqref{label_046} under the restriction $A_1 = 0$ and $\gamma_1 = 1$ correspond exactly to the formulation of the Information Theoretical Exact Method (ITEM) found in \cite{ref_001}. Further setting $\mu = 0$ produces the online version of the Optimized Gradient Method (see \cite{ref_007} for a detailed discussion on online vs. offline mode in OGM). Note that the first gradient step iteration can be discarded when $A_1 = 0$ thereby increasing $k$ by $1$ in the \emph{right-hand sides} of \eqref{label_049}, \eqref{label_053} and \eqref{label_054}, respectively.

If $\mu > 0$ we can have $\gamma_1 = 2 \mu r A_1$, in which case the update rules in \eqref{label_043}, \eqref{label_044} and \eqref{label_046} become for all $k \geq 1$ those of the Triple Momentul Method~\cite{ref_019}, namely
\begin{align}
A_{k + 1} &= (1 - \sqrt{q})^{-2} A_k,\\
y_{k + 1} &= \frac{1 - \sqrt{q}}{1 + \sqrt{q}} \left(y_k - \frac{1}{L_f} B^{-1} g_k \right) + \frac{2 \sqrt{q}}{1 + \sqrt{q}} v_k,\\
v_{k + 1} &= (1 - \sqrt{q}) v_k + \sqrt{q} \left( y_{k + 1} - \frac{1}{\mu} B^{-1} g_{k + 1} \right).
\end{align}
Every TMM iteration requires the value of the gradient at the \emph{previous} iteration and therefore we cannot discard the first gradient step. This also explains why in this case $w_1$ has to be omitted from the aggregate bound $W_k$ in \eqref{label_013} of the estimate function $\psi_k$ in \eqref{label_014}.

The relationship between Algorithm~\ref{label_047} and the existing optimal methods is summarized in Table~\ref{label_055}.

\begin{table}[h]
\caption{Several optimal methods are instances of Algorithm~\ref{label_047} with restrictions applied}
\label{label_055}
\centering
\begin{tabular}{lcccc} \toprule
Algorithm & \multicolumn{4}{c}{Restriction} \\
& $\mu = 0$ & $\mu > 0$ & $A_1 = 0$ & $A_1 = 2 \mu r \gamma_1$ \\ \midrule
OGM~\cite{ref_011} & yes & no & yes & no \\
TMM~\cite{ref_019} & no & yes & no & yes \\
ITEM~\cite{ref_001} & no & no & yes & no \\ \bottomrule
\end{tabular}
\end{table}

\subsection{Augmentation}

It is possible to obtain Lyapunov (potential) functions from the estimate functions by means of augmentation~\cite{ref_010,ref_003}. We define the augmented estimate functions for all $x \in \mathbb{E}$, $g \in \mathbb{E}^*$ and $k \geq 1$ as
\begin{equation} \label{label_056}
\bar{\psi}_k(x, g) = \psi_k(x, g) + (A_k - A_1) (f^* - W_k(x^*, 0)).
\end{equation}
The corresponding augmented estimate sequence gap becomes $\bar{\Gamma}_k \overset{\operatorname{def}}{=} \bar{\psi}_k^* - A_k f(y_k)$, $k \geq 1$. The functions $w_k$, $k \geq 2$, are lower bounds on the objective at $(x^*, 0)$ and hence the augmented estimate sequence property, stated as $\bar{\Gamma}_k \geq 0$, constitutes a relaxation of the ESP in \eqref{label_018}. We establish a connection between the estimate sequence and potential functions using the following result.
\begin{lemma} \label{label_057}
The augmented estimate sequence gap can be expressed as the difference between two gap sequence terms:
\begin{equation}
\bar{\Gamma}_k = \mathcal{D}_1 - \mathcal{D}_k, \quad k \geq 1,
\end{equation}
where the gap sequence terms are defined as
\begin{equation} \label{label_058}
\mathcal{D}_k \overset{\operatorname{def}}{=} A_k(f(y_k) - f^* - \hat{w}_k(x^*)) + \frac{\gamma_k}{2} \| v_k - x^* \|^2 , \quad k \geq 1.
\end{equation}
\end{lemma}
\begin{proof}
Herein we consider all $k \geq 1$. From \eqref{label_056} we get
\begin{equation} \label{label_059}
\bar{\Gamma}_k = \psi_k^* + (A_k - A_1)(f^* - W_k(x^*, 0)) - A_k f(y_k).
\end{equation}
From the definition of the estimate function in \eqref{label_014} as well as the canonical form in \eqref{label_016} we have
\begin{equation} \label{label_060}
\begin{gathered}
\psi_k^* + \frac{\gamma_k}{2} \| x^* - v_k \|^2 \overset{\eqref{label_016}}{=} \psi_k(x^*, 0) \overset{\eqref{label_014}}{=} (A_k - A_1) W_k(x^*, 0) + A_k \hat{w}_k(x^*) \\ + A_1 (f(y_1) - \hat{w}_1(x^*)) + \frac{\gamma_1}{2} \| x^* - v_1 \|^2.
\end{gathered}
\end{equation}
Combining \eqref{label_059} and \eqref{label_060} yields
\begin{equation}
\begin{gathered}
\bar{\Gamma}_k = (A_k - A_1)f^* + A_k \hat{w}_k(x^*) + A_1 (f(y_1) - \hat{w}_1(x^*)) \\ + \frac{\gamma_1}{2} \| x^* - v_1 \|^2 - \frac{\gamma_k}{2} \| x^* - v_k \|^2 - A_k f(y_k) \overset{\eqref{label_058}}{=} \mathcal{D}_1 - \mathcal{D}_k.
\end{gathered}
\end{equation}
\end{proof}
The gap sequence terms are the potential functions that were previously used in \cite{ref_001} to analyze the convergence of ITEM and TMM, but not to derive them. In the sequel, we study the behavior of this sequence when their parameters are updated based on our estimate function framework
\begin{proposition} \label{label_061}
A first-order method that updates the state parameters according to \eqref{label_017} and \eqref{label_026} satisfies
\begin{equation} \label{label_062}
\begin{gathered}
\mathcal{D}_{k} - \mathcal{D}_{k + 1} \geq \frac{\gamma_{k}}{2} \| v_{k} - y_{k + 1} \|^2 - \frac{\gamma_{k + 1}}{2} \| v_{k + 1} - y_{k + 1} \|^2 \\ + r A_k \langle g_{k + 1}, x_k - y_{k + 1} \rangle + \frac{r A_{k + 1}}{L_f} \| g_{k + 1} \|_*^2, \quad k \geq 1.
\end{gathered}
\end{equation}
\end{proposition}
\begin{proof}
Our reasoning follows the same structure as the proofs of \cite[Theorem 1]{ref_009} and \cite[Theorem 1]{ref_010}. Within this proof we consider all $k \geq 1$. We define the residual $\mathcal{R}_k(y, g) \overset{\operatorname{def}}{=} f(y) - w_k(T_{L_f}(y), g)$ for all $y \in \mathbb{E}$. We have $\mathcal{R}_{k + 1}(y_k, f'(y_{k})) \geq 0$ and $\mathcal{R}_{k + 1}(x^*, 0) \geq 0$. The non-negativity of $A_k$ and $a_{k + 1}$ implies that
\begin{equation} \label{label_063}
A_k \mathcal{R}_{k + 1}(y_k, f'(y_{k})) + a_{k + 1} \mathcal{R}_{k + 1}(x^*, 0) \geq 0.
\end{equation}
Expanding and regrouping terms in \eqref{label_063} gives
\begin{equation} \label{label_064}
\begin{gathered}
A_k f(y_k) + a_{k + 1} f^* \geq A_{k + 1} f(y_{k + 1}) \\+ \left\langle g_{k + 1}, A_k \left(y_k - \frac{1}{L_f} B^{-1} g_k \right) + a_{k + 1} x^* - A_{k + 1} y_{k + 1} \right\rangle \\+ \frac{\mu r A_k}{2} \| x_k - x_{k + 1} \|^2 + \frac{\mu r a_{k + 1}}{2} \| x^* - x_{k + 1} \|^2 + \frac{A_{k}}{2 L_f} \| g_k \|_*^2 + \frac{A_{k + 1}}{2 L_f} \| g_{k + 1} \|_*^2.
\end{gathered}
\end{equation}
Adding $\frac{\gamma_k}{2} \| x^* - v_k \|^2 - \frac{\gamma_{k + 1}}{2} \| x^* - v_{k + 1} \|^2$ to both sides of \eqref{label_064} and arranging terms gives
\begin{equation} \label{label_065}
\begin{gathered}
\mathcal{D}_k - \mathcal{D}_{k + 1} \geq \frac{\gamma_k}{2} \| x^* - v_k \|^2 - \frac{\gamma_{k + 1}}{2} \| x^* - v_{k + 1} \|^2 + \frac{A_{k + 1}}{L_f} \| g_{k + 1} \|_*^2 \\+ \langle g_{k + 1}, A_k x_k + a_{k + 1} x^* - A_{k + 1} y_{k + 1} \rangle \\+ \frac{\mu r A_k}{2} \left(\| x_k - x_{k + 1} \|^2 - \| x_k - x^* \|^2 \right) + \frac{\mu r (A_{k + 1} + a_{k + 1})}{2} \| x_{k + 1} - x^* \|^2.
\end{gathered}
\end{equation}
We regroup the terms in the right-hand side of \eqref{label_065} around powers of $x^*$ and obtain
\begin{equation}
\mathcal{D}_k - \mathcal{D}_{k + 1} \geq \mathcal{C}_k^{(2)} \| x^* \|^2 + \left\langle B \mathcal{C}_k^{(1)}, x^* \right\rangle + \mathcal{C}_k^{(0)} ,
\end{equation}
where $\mathcal{C}_k^{(2)} \overset{\operatorname{def}}{=} \gamma_k - \gamma_{k + 1} - \mu r A_k + \mu r (A_{k + 1} + a_{k + 1})$,
\begin{equation} \label{label_066}
\mathcal{C}_k^{(1)} \overset{\operatorname{def}}{=} \gamma_{k + 1} v_{k + 1} - \gamma_k v_k + a_{k + 1} B^{-1} g_{k + 1}
+ \mu r A_k x_k - \mu r (A_{k + 1} + a_{k + 1}) x_{k + 1},
\end{equation}
\begin{equation} \label{label_067}
\begin{gathered}
\mathcal{C}_k^{(0)} \overset{\operatorname{def}}{=} \mu r A_{k + 1} \| x_{k + 1} \|^2 - \mu r A_k \langle B x_k, x_{k + 1} \rangle + \langle g_{k + 1}, A_k x_k - A_{k + 1} y_{k + 1} \rangle \\ + \frac{\gamma_k}{2} \| v_k \|^2 - \frac{\gamma_{k + 1}}{2} \| v_{k + 1} \|^2 + \frac{A_{k + 1}}{L_f} \| g_{k + 1} \|_*^2.
\end{gathered}
\end{equation}
The update in \eqref{label_017} implies that $\mathcal{C}_k^{(2)} = 0$ whereas \eqref{label_026} is equivalent to \eqref{label_028}. Combining \eqref{label_028} with \eqref{label_066} gives $\mathcal{C}_k^{(1)} = 0$.

Based on \eqref{label_003}, we perform the expansion:
\begin{equation} \label{label_067_1}
\begin{gathered}
\mu r A_{k + 1} \| x_{k + 1} \|^2 - \mu r A_k \langle B x_k, x_{k + 1} \rangle = \mu r A_{k + 1} \|y_{k + 1} \|^2 + \frac{q r A_{k + 1}}{L_f} \| g_{k + 1} \|_*^2 \\- 2 q r A_{k + 1} \langle g_{k + 1}, y_{k + 1}\rangle - \mu r A_k \langle B x_k, y_{k + 1} \rangle + q r A_k \langle g_{k + 1}, x_k \rangle.
\end{gathered}
\end{equation}
We also use the form in \eqref{label_029} of \eqref{label_026} to perform another expansion:
\begin{equation} \label{label_069}
\begin{gathered}
\frac{\gamma_{k}}{2} \| v_{k} - y_{k + 1} \|^2 - \frac{\gamma_{k + 1}}{2} \| v_{k + 1} - y_{k + 1} \|^2 = \frac{\gamma_k}{2} \| v_k \|^2 - \frac{\gamma_{k + 1}}{2} \| v_{k + 1} \|^2 - \mu r a_{k + 1} \| y_{k + 1} \|^2 \\
- \langle \bar{a}_{k + 1} g_{k + 1}, y_{k + 1} \rangle + \mu r (A_{k + 1} + a_{k + 1}) \| y_{k + 1} \|^2 - \mu r A_k \langle B x_k, y_{k + 1} \rangle.
\end{gathered}
\end{equation}
Putting together \eqref{label_067}, \eqref{label_067_1} and \eqref{label_069} we obtain
\begin{equation}
\begin{gathered}
\mathcal{C}_k^{(0)} = \frac{\gamma_{k}}{2} \| v_{k} - y_{k + 1} \|^2 - \frac{\gamma_{k + 1}}{2} \| v_{k + 1} - y_{k + 1} \|^2 + r A_k \langle g_{k + 1}, x_k - y_{k + 1} \rangle \\ + (r A_k - 2 q r A_{k + 1} + \bar{a}_{k + 1} - A_{k + 1}) \langle g_{k + 1}, y_{k + 1} \rangle + \frac{r A_{k + 1}}{L_f} \| g_{k + 1} \|_*^2.
\end{gathered}
\end{equation}
Finally, \eqref{label_027} and $1 + rq = r$ imply that $r A_k - 2 q r A_{k + 1} + \bar{a}_{k + 1} - A_{k + 1} = 0$.
\end{proof}
The right-hand side of \eqref{label_062} is identical to the corresponding one in the simple inequality \eqref{label_032}. Using the proof mechanism of Theorem~\ref{label_033}, we immediately obtain the Lyapunov (potential) function counterpart of Theorem~\ref{label_033}.
\begin{corollary} \label{label_070}
For any first order method that uses the updates \eqref{label_017} and \eqref{label_026}, Proposition~\ref{label_061} implies that
\begin{equation}
\begin{gathered}
\mathcal{D}_{k} - \mathcal{D}_{k + 1} \geq
\left( \frac{r A_{k + 1}}{L_f} - \frac{\bar{a}_{k + 1}^2}{2 \gamma_{k + 1}} \right)
\left( \| g_{k + 1} \|_*^2 - \frac{\mu^2 \gamma_k \gamma_{k + 1}}{\bar{\gamma}^2_{k + 1}} \| v_k - y_{k + 1} \|^2 \right) \\
+ \left\langle \frac{\bar{\gamma}_{k + 1}}{\gamma_{k + 1}} g_{k + 1} + \mu B \left( \frac{\gamma_k}{\bar{\gamma}_{k + 1}} (v_k - y_{k + 1}) - \frac{\mu}{2 \gamma_{k + 1}} Y_{k + 1} \right), Y_{k + 1} \right\rangle, \quad k \geq 1,
\end{gathered}
\end{equation}
where $\bar{a}_{k + 1}$, $\bar{\gamma}_{k + 1}$ and $Y_{k + 1}$ are defined using \eqref{label_027}, \eqref{label_034} and \eqref{label_035}, respectively.
\end{corollary}
Unlike in \cite{ref_007}, where the potential function analysis yielded better guarantees than the estimate functions, here we see that identical results are obtained using both approaches, noting that the estimate sequence is superior both in terms of making algorithm construction straightforward but also, as we shall see in the sequel, in adding memory and adaptivity to the algorithm.

\section{An Optimized Gradient Method with Memory for strongly convex objectives}

To apply the bundle to our generalized Optimized Gradient Gethod in Algorithm~\ref{label_047}, we first need to re-center the primal term in the bound $w_k(x, g)$ defined in \eqref{label_008} around the fixed starting point $v_1$ for all $x \in \mathbb{E}$, $g \in \mathbb{E}^*$, $k \geq 1$ as
\begin{equation} \label{label_071}
\begin{gathered}
w_k(x, g) = \bar{h}_k + \langle \bar{g}_k, x \rangle + \frac{\mu r}{2} \| x - v_1 \|^2 + \frac{1}{2 L_f} \| g \|_*^2,
\end{gathered}
\end{equation}
where
\begin{align}
\bar{h}_k & \overset{\operatorname{def}}{=} f(y_k) - \langle g_k, y_k \rangle + \frac{\mu r}{2} \left( \| x_k \|^2 - \| v_1 \|^2 \right) + \frac{1}{2 L_f} \| g_k \|_*^2, \label{label_072}\\
\bar{g}_k & \overset{\operatorname{def}}{=} g_k + \mu r B (v_1 - x_k). \label{label_073}
\end{align}
More details on how methods with memory can be constructed to take advantage of strong convexity, including the re-centering of the bounds, can be found in \cite{ref_006}.

Next, at the beginning of every iteration $k \geq 1$, we consider a model $(H_k, G_k)$, $H_k \in \mathbb{R}^{m_k}$, $G_k \in \mathbb{R}^{m_k} \times \mathbb{E}^*$, containing $m_k$ entries, such that
\begin{equation} \label{label_074}
f(y) \geq \max\{H_k + G_k^T y \} + \frac{\mu r}{2} \| x - v_1 \|^2 + \frac{1}{2 L_f} \| f'(y) \|_*^2 , \quad y \in \mathbb{R}^n,
\end{equation}
where $\max$ denotes element-wise maximum. The reformulation in \eqref{label_071} suggests a means of constructing such a model.

\begin{proposition}
If for every $i \inn{m_k}$ there exists a vector of weights $\mathcal{T}_i \in \Delta_{k}$ (where $\Delta_p$ is the p-dimensional simplex for any $p \geq 1$) that constructs our model entry $i$ as $(H)_i = \sum_{j = 1}^{k} (\mathcal{T}_i)_j \bar{h}_j$ and $(G)_i = \sum_{j = 1}^{k} (\mathcal{T}_i)_j \bar{g}_j$, then \eqref{label_074} holds.
\end{proposition}
\begin{proof}
The global lower bound property of $w_k$ can be expressed using \eqref{label_072} and \eqref{label_073}, respectively, in the following form:
\begin{equation} \label{label_075}
f(y) \geq \bar{h}_k + \langle \bar{g}_k, x \rangle + \frac{\mu r}{2} \| x - v_1 \|^2 + \frac{1}{2 L_f} \| f'(y) \|_*^2, \quad y \in \mathbb{E}, \quad k \geq 1.
\end{equation}
Because $(\mathcal{T}_i)_j \geq 0$ for all $i \inn{m_k}$, $j \inn{k}$, we can aggregate \eqref{label_075} for all $y \in \mathbb{E}$, $i \inn{m_k}$ and $k \geq 1$ as:
\begin{equation}
\begin{gathered}
\sum_{j = 1}^{k} (\mathcal{T}_i)_j f(y) \geq \sum_{j = 1}^{k} (\mathcal{T}_i)_j \bar{h}_j + \left\langle \sum_{j = 1}^{k} (\mathcal{T}_i)_j \bar{g}_j, x \right\rangle
\\ + \sum_{j = 1}^{k} (\mathcal{T}_i)_j \left( \frac{\mu r}{2} \| x - v_1 \|^2 + \frac{1}{2 L_f} \| f'(y) \|_*^2 \right).
\end{gathered}
\end{equation}
Using $ \sum_{j = 1}^{k} (\mathcal{T}_i)_j = 1$ for all $i \inn{m_k}$ yields the desired result.
\end{proof}
We generalize the estimate function in \eqref{label_014} to include the arbitrary weights $\lambda \in \Delta_{m_k}$ and the possibility to increase the convergence guarantee $A_k$ in variable $A$ as follows:
\begin{equation} \label{label_076}
\begin{gathered}
\psi_k(x, g, A, \lambda) = (A - A_1) \langle H_k + G_k^T x, \lambda \rangle + \frac{\mu r (A - A_1)}{2} \| x - v_1 \|^2 \\ + \frac{A - A_1}{2 L_f} \| g \|_*^2 + A \hat{w}_k(x) + A_1 (f(y_1) - \hat{w}_1(x)) + \frac{\gamma_1}{2} \| x - v_1 \|^2, \quad k \geq 1.
\end{gathered}
\end{equation}
We still need to satisfy the ESP in \eqref{label_018}, and hence consider the normalized variant of \eqref{label_076} implicitly parameterized by $A$ and $\lambda$ as $\omega_k(x, g, A, \lambda) \overset{\operatorname{def}}{=} \frac{1}{A} \psi_k(x, g, A, \lambda)$. We re-center $\hat{w}_k$ around $y_1$ and obtain
\begin{equation} \label{label_077}
\hat{w}_k(x) = \hat{h}_k + \langle \hat{g}_k, x \rangle + \frac{\mu r}{2} \| x - v_1 \|^2, \quad k \geq 1,
\end{equation}
where
\begin{equation} \label{label_078}
\hat{h}_k \overset{\operatorname{def}}{=} \frac{\mu r}{2} ( \| x_k \|^2 - \| v_1 \|^2 ) + \frac{1}{2 L_f} \| g_k \|_*^2, \quad
\hat{g}_k \overset{\operatorname{def}}{=} \mu r B ( v_1 - x_k ) , \quad k \geq 1.
\end{equation}
In the remainder of this section we consider all $k \geq 1$, unless specified otherwise.

Using the form in \eqref{label_077}, we refactor the normalized estimate function (NEF) $\omega_k$ to take the form
\begin{equation} \label{label_079}
\begin{gathered}
\omega_k(x, g, A, \lambda) = \frac{A - A_1}{A} \left( \langle H_k + G_k^T x, \lambda \rangle + \frac{1}{2 L_f} \| g \|_*^2 \right) \\ + \hat{h}_k + \langle \hat{g}_k, x \rangle + \frac{A_1}{A} \left( f(y_1) - \hat{h}_1 - \langle \hat{g}_1, x \rangle \right) + \frac{\gamma(A)}{2 A} \| x - v_1 \|^2,
\end{gathered}
\end{equation}
where $\gamma(A) \overset{\operatorname{def}}{=} \gamma_1 + 2 \mu r (A - A_1)$. Strong duality holds for $\omega_k$ and we have the optimum over all variables except $A$ given by
\begin{equation} \label{label_080}
\omega_k^*(A) = \max_{\lambda \in \Delta_{m_k}} \min_{x \in \mathbb{E}, g \in \mathbb{E}^*} \omega_k(x, g) = \max_{\lambda \in \Delta_{m_k}} \omega_k(v_k(A, \lambda), 0, A, \lambda),
\end{equation}
where the optimal point $v_k(A, \lambda)$ along with the subexpressions $\rho_k(A, \lambda)$ and $\nu_k(A)$ are, respectively, defined as
\begin{gather}
v_k(A, \lambda) \overset{\operatorname{def}}{=} v_1 - \frac{A}{\gamma(A)} B^{-1} \rho_k(A, \lambda), \label{label_081}\\
\rho_k(A, \lambda) \overset{\operatorname{def}}{=} \frac{A - A_1}{A} G_k \lambda + \nu_k(A), \quad \nu_k(A) \overset{\operatorname{def}}{=} \hat{g}_k - \frac{A_1}{A} \hat{g}_1. \label{label_082}
\end{gather}
The optimal value is thus given by
\begin{equation} \label{label_079_star}
\begin{gathered}
\omega_k(v_k(A, \lambda), 0, A, \lambda) = \frac{A - A_1}{A} \langle H_k, \lambda \rangle + \hat{h}_k \\ + \frac{A_1}{A} (f(y_1) - \hat{h}_1) + \langle \rho_k(A, \lambda), v_1 \rangle - \frac{A}{2 \gamma(A)} \| \rho_k(A, \lambda) \|_*^2.
\end{gathered}
\end{equation}
Expanding \eqref{label_079_star} using \eqref{label_081} and \eqref{label_082} we obtain
\begin{equation} \label{label_084}
\omega_k^*(A) = \max_{\lambda \in \Delta_{m_k}} - \frac{P(A)}{2} \lambda^T Q_k \lambda + \langle S_k(A), \lambda \rangle + R_k(A),
\end{equation}
where
\begin{gather}
Q_k \overset{\operatorname{def}}{=} G_k^T B^{-1} G_k, \quad P(A) \overset{\operatorname{def}}{=} \frac{(A - A_1)^2}{A \gamma(A)}, \label{label_085}\\
S_k(A) \overset{\operatorname{def}}{=} \frac{A - A_1}{A} (H_k + G_k^T v_1) - \frac{A - A_1}{\gamma(A)} G_k^T B^{-1} \nu_k(A), \label{label_086}\\
R_k(A) \overset{\operatorname{def}}{=} \hat{h}_k + \frac{A_1}{A} (f(y_1) - \hat{h}_1) + \langle \nu_k(A), v_1 \rangle - \frac{A}{2 \gamma(A)} \| \nu_k(A) \|_*^2 . \label{label_087}
\end{gather}
The auxiliary problem in \eqref{label_084} is a Quadratic Program under a simplex constraint that does not admit a known closed form solution. Many methods exist that solve \eqref{label_084} approximately but, for the purpose of devising a convergence guarantee adjustment procedure, we temporarily assume that the solution is exact and denote it by $\lambda(A)$.
The problem of maximizing $A$ while maintaining the ESP in \eqref{label_018} becomes that of finding the root of the normalized gap function, given by
\begin{align}
\phi_k(A) &\overset{\operatorname{def}}{=} \frac{1}{A} \Gamma_k(A) = \omega_k^*(A) - f(y_k) \label{label_088} \\
&= - \frac{P(A)}{2} \lambda^T(A) Q_k \lambda(A) + \langle S_k(A), \lambda(A) \rangle + R_k(A) - f(y_k). \label{label_089}
\end{align}
As with our previous gradient methods with memory, we turn to Newton's rootfinding method, with steps given by $A^{(t + 1)} = A^{(t)} - \frac{\phi_k(A^{(t)})}{\phi_k'(A^{(t)})}$. Under our exactness assumption on $\lambda(A)$, we can again use Danskin's lemma in \eqref{label_089} to yield the gap function derivative as
\begin{equation} \label{label_090}
\phi_k'(A) = - \frac{P'(A)}{2} \lambda(A)^T Q_k \lambda(A) + \langle S_k'(A), \lambda(A) \rangle + R_k'(A),
\end{equation}
where the subexpressions $P'(A)$, $S_k'(A)$ and $R_k'(A)$, respectively, take the form
\begin{equation}
P'(A) = \left( \frac{(A - A_1)^2}{A \gamma(A)} \right)' = \frac{(A - A_1) ( A \gamma_1 + A_1 \gamma(A))}{A^2 \gamma^2(A)},
\end{equation}
\begin{equation}
S_k'(A) = \frac{A_1}{A^2}(H_k + G_k^T v_1) - \frac{1}{\gamma^2(A)} G_k^T B^{-1} \left( \gamma_1 \hat{g}_k - \left(\frac{A_1 \gamma(2 A)}{A^2} - 2 \mu r \right) A_1 \hat{g}_1 \right),
\end{equation}
\begin{equation}
\begin{gathered}
R_k'(A) = -\frac{A_1}{A^2} \left( f(y_1) - \hat{h}_1 - \langle \hat{g}_1, v_1 \rangle \right) \\ - \frac{1}{2 \gamma^2(A)} \left( \gamma(0) \| \hat{g}_k \|_*^2 - \frac{A_1^2 \gamma(2 A)}{A^2} \| \hat{g}_1 \|_*^2 + 4 \mu r A_1 \langle \hat{g}_1, B^{-1} \hat{g}_k \rangle \right) .
\end{gathered}
\end{equation}
The above derivative expressions were obtained using the equalities $\left( A \gamma(A) \right)' = \gamma(2 A)$ and $\left(\frac{A}{\gamma(A)} \right)' = \frac{\gamma(0)}{\gamma^2(A)}$.

In practice, however, no algorithm can ensure that $\lambda(A)$ is an exact solution. We retain the Newton update but terminate if either $\phi_k(A) < 0$ or $\phi_k'(A) \geq 0$. This way the ESP is always satisfied and the convergence guarantees can only increase via the Netwon update. The procedure for dynamically adjusting the convergence guarantee is listed in Algorithm~\ref{label_091}. This constitutes a middle scheme called by our main algorithm, the outer scheme. In turn, it calls the auxiliary problem solver which constitutes the inner scheme. The explicit parameters are the memory model, completely described by $H_k$ and $G_k$, the efficiently updated matrix $Q_k$, the current partial bound at the optimum defined by $\hat{h}_k$ and $\hat{g}_k$ as well as the current convergence guarantee $A^{(0)}$ and the initial $\lambda^{(0)}$ used the auxiliary problem solver. Implicit parameters that do not change while the outer scheme is running describe the initial state, namely $A_1$, $\gamma_1$, $v_1$, $\hat{h}_1$ and $\hat{g}_1$. The Newton scheme outputs the final values of $A_k$ and $\lambda_k$ that will be used to update the model in the next outer iteration.

\begin{algorithm}[h!]
\caption{Newton$(H_k, G_k, Q_k, \hat{h}_k, \hat{g}_k, \lambda^{(0)}, A^{(0)})$}
\footnotesize
\label{label_091}
\begin{algorithmic}
\STATE $\lambda_{\operatorname{valid}} := \lambda^{(0)}$
\STATE $A_{\operatorname{valid}} := A := A^{(0)}$
\FOR{$t = 1, \ldots{}, N$}
\STATE Compute $P(A)$, $S_{k}(A)$, $R_{k}(A)$ according to \eqref{label_085}, \eqref{label_086} and \eqref{label_087}
\STATE Compute $\lambda$ as an approximate solution to \eqref{label_084} with starting point $\lambda^{(0)}$
\STATE Compute $\phi_k(A)$ and $\phi_k'(A)$ according to \eqref{label_088} and \eqref{label_090}
\IF{$\phi_k(A) < 0$}
\STATE Break from loop
\ENDIF
\STATE $\lambda_{\operatorname{valid}} := \lambda$
\STATE $A_{\operatorname{valid}} := A$
\IF{$\phi_k'(A) \geq 0$}
\STATE Break from loop
\ENDIF
\STATE $A := A - \frac{\phi_k(A)}{\phi_k'(A)}$
\ENDFOR
\RETURN $\lambda_{\operatorname{valid}}, A_{\operatorname{valid}}$
\end{algorithmic}
\end{algorithm}

The first iteration of the outer scheme, corresponding to $k = 0$, is a gradient step, necessary to define the starting point $x_1$ when $A_1 > 0$, as previously argued. The estimate sequence optimum $v_1$ can be chosen to be $x_1$ or any other point, possibly specified by the user alongside $x_0$. Setting $A = A_1$ gives a normalized estimate function (NEF) optimum of $\omega^*_1(A_1) = f(y_1)$ leading to a normalized gap $\phi_1(A_1) = 0$. This choice of $A = A_1$ and $\lambda = 1$ is optimal and there is no need to call the middle an inner schemes, which also cannot be employed due to the absence of a model.

Next, when $k = 1$, we are ready to initialize our model. We choose the first entry to be $H_2 = \bar{h}_2$ and $G_2 = \bar{g}_2$, with $\bar{h}_2$ and $\bar{g}_2$ given by \eqref{label_072} and \eqref{label_073}, respectively. Likewise, we obtain $\hat{h}_2$ and $\hat{g}_2$ according to \eqref{label_078}. We call the middle (Newton) scheme in Algorithm~\ref{label_091} with parameters $A_2^{(0)} = A_1 + a_2$ and $\lambda_2^{(0)} = 1$. At this point, our NEF optimum is given by $\omega^*_2\left(A_2^{(0)}\right) = -\frac{1}{2} P\left(A_2^{(0)}\right) Q_2 + S_2\left(A_2^{(0)}\right) + R_2\left(A_2^{(0)}\right)$ with $P$, $Q_2$, $S_2$ and $R_2$, respectively given by \eqref{label_085}, \eqref{label_086} and \eqref{label_087}. At this initial stage, the second iteration has the same state as the second iteration of the generalized OGM in Algorithm~\ref{label_047} and via Theorem~\ref{label_033} we have that $A_2^{(0)} \phi_2\left(A_2^{(0)}\right) \geq A_1 \phi(A_1) = 0$. The Newton middle scheme thus starts with a valid estimate sequence property that it preserves up to termination.

From the third iteration onward, namely $k \geq 2$, we always have stored in memory the previous NEF $\omega_k$ defined by the bounds $\hat{h}_k$ and $\hat{g}_k$ along with the compacted version of the model, defined by $\tilde{h}_k \overset{\operatorname{def}}{=} \langle H_{k}, \lambda_{k} \rangle$ and $\tilde{g}_k \overset{\operatorname{def}}{=} G_k \lambda_k$. This function has the expression
\begin{equation}
\begin{gathered}
\omega_k(x, g) = \frac{A_k - A_1}{A_k} \left( \tilde{h}_k + \langle \tilde{g}_k, x \rangle + \frac{1}{2 L_f} \| g \|_*^2 \right) \\ + \hat{h}_k + \langle \hat{g}_k, x \rangle + \frac{A_1}{A_k} \left( f(y_1) - \hat{h}_1 - \langle \hat{g}_1, x \rangle \right) + \frac{\gamma(A_k)}{2 A_k} \| x - v_1 \|^2 .
\end{gathered}
\end{equation}
Now, when constructing the new NEF $\omega_{k + 1}$ according to \eqref{label_079}, we can ensure that the first two model entries have the following structure:
\begin{equation} \label{label_092}
H_{k + 1}^{(1)} = \tilde{h}_k, \quad H_{k + 1}^{(2)} = \bar{h}_{k + 1}, \quad G_{k + 1}^{(1)} = \tilde{g}_k, \quad G_{k + 1}^{(2)} = \bar{g}_{k + 1}.
\end{equation}
We call the Newton method with $A_{k + 1}^{(0)} = A_k + a_{k + 1}$ and with the inner scheme starting point $\lambda^{(0)}_{k + 1} = \frac{1}{A_k + a_{k + 1} - A_1}(A_k - A_1, a_{k + 1}, 0, ..., 0)^T$. We thus have
\begin{equation}
\begin{gathered}
A_{k + 1}^{(0)} \omega_{k + 1} \left(x, g, A_{k + 1}^{(0)}, \lambda^{(0)}_{k + 1} \right) - A_{k + 1}^{(0)} \hat{w}_{k + 1}(x) \\= A_k \omega_{k} \left(x, g, A_{k}, \lambda_{k + 1} \right) - A_k \hat{w}_k(x) + a_{k + 1} w_{k + 1}(x, g), \quad x \in \mathbb{E}, \quad g \in \mathbb{E}^*.
\end{gathered}
\end{equation}
We have just shown that \eqref{label_023} and its expansion in \eqref{label_024} hold. Applying again Theorem~\ref{label_033} we have that $\left( A^{(0)}_{k + 1} \right) \phi_{k + 1}\left(A^{(0)}_{k + 1} \right) \geq A_k \phi(A_k) \geq 0$. At the start of the middle method the gap is non-negative and this property is preserved thereby ensuring the convergence of our method with the same worst-case guarantees as the generalized OGM in Algorithm~\ref{label_047}. Our Optimized Gradient Method with Memory is listed in Algorithm~\ref{label_047m}.

\begin{algorithm}[h!]
\caption{An Optimized Gradient Method with Memory}
\footnotesize
\label{label_047m}
\begin{algorithmic}[1]
\STATE \textbf{Input:} $B \succ 0$, $x_0 \in \mathbb{E}$, $\mu \geq 0$, $L_f > 0$, $A_1 \geq 0$, $\gamma_1 > 0$, $T \inn{\infty}$
\STATE $y_1 = x_0$
\STATE $f_{1} = f(y_{1})$, $g_{1} = f'(y_{1})$\\[1mm]
\STATE $x_1 = y_1 - \frac{1}{L_f} B^{-1} g_1$
\STATE Choose $v_1$ from among $x_1$ (recommended), $x_0$, or an additional input point
\STATE $\hat{h}_1 = \frac{\mu r}{2}\left( \|x_1\|^2 - \|v_1\|^2 \right) + \frac{1}{2 L_f} \| g_{1} \|_*^2$
\STATE $\hat{g}_1 = \mu r B ( v_1 - x_1 )$
\FOR{$k = 1,\ldots{},T$}
\STATE $a_{k + 1} = \frac{1}{L_f - \mu} \left( \gamma_k + \mu A_k + \sqrt{\gamma_k (\gamma_k + 2 L_f A_k)} \right)$ \\[1mm]
\STATE $\gamma_{k + 1} := \gamma_k + 2 \mu r a_{k + 1}$
\STATE $\bar{a}_{k + 1} = (1 + 2 q r ) a_{k + 1} + q r A_k$
\STATE $\bar{\gamma}_{k + 1} = \gamma_{k + 1} - \mu \bar{a}_{k + 1}$
\STATE $y_{k + 1} = (r A_k \bar{\gamma}_{k + 1} + \bar{a}_{k + 1} \gamma_k)^{-1} \left( r A_k \bar{\gamma}_{k + 1} x_k + \bar{a}_{k + 1} \gamma_k v_k \right)$\\[1mm]
\STATE $f_{k + 1} = f(y_{k + 1})$, $g_{k + 1} := f'(y_{k + 1})$
\STATE $x_{k + 1} = y_{k + 1} - \frac{1}{L_f} B^{-1} g_{k + 1}$
\STATE $\hat{h}_{k + 1} = \frac{\mu r}{2}\left( \|x_{k + 1}\|^2 - \|v_1\|^2 \right) + \frac{1}{2 L_f} \| g_{k + 1} \|_*^2$
\STATE $\bar{h}_{k + 1} = \hat{h}_{k + 1} + f_{k + 1} - \langle g_{k + 1}, y_{k + 1} \rangle$
\STATE $\hat{g}_{k + 1} = \mu r B (v_1 - x_{k + 1} )$
\STATE $\bar{g}_{k + 1} = \hat{g}_{k + 1} + g_{k + 1} $
\IF {$k = 1$}
\STATE $H_2 = \bar{h}_{2}$, $G_2 = \bar{g}_2$, $Q_2 = \| \bar{g}_2 \|_*^2$, $\lambda^{(0)}_{2} = 1$
\ELSE
\STATE Generate $H_{k + 1}$ and $G_{k + 1}$ to satisfy \eqref{label_092}
\STATE Generate $Q_{k + 1}$ to equal $G_{k + 1}^T B^{-1} G_{k + 1}$
\STATE $\lambda^{(0)}_{k + 1} = \frac{1}{A_k + a_{k + 1} - A_1}(A_k - A_1, a_{k + 1}, 0, ..., 0)^T$
\ENDIF
\STATE $\lambda_{k + 1}, A_{k + 1} = \operatorname{Newton}(H_{k + 1}, G_{k + 1}, Q_{k + 1}, \hat{h}_{k + 1}, \hat{g}_{k + 1}, \lambda^{(0)}_{k + 1}, A_k + a_{k + 1})$

\STATE $\gamma_{k + 1} := \gamma_1 + 2 \mu r (A_{k + 1} - A_1)$
\STATE $\tilde{h}_{k + 1} = \langle H_{k + 1}, \lambda_{k + 1} \rangle$
\STATE $\tilde{g}_{k + 1} = G_{k + 1} \lambda_{k + 1}$
\STATE $\rho_{k + 1} = \left(1 - \frac{A_1}{A_{k + 1}} \right) \tilde{g}_{k + 1} + \hat{g}_{k + 1} - \frac{A_1}{A_{k + 1}} \hat{g}_1$
\STATE $v_{k + 1} = v_1 - \frac{A_{k + 1}}{\gamma_{k + 1}} B^{-1} \rho_{k + 1}$
\ENDFOR
\end{algorithmic}
\end{algorithm}

\section{Improving the guarantees in iterate space on composite problems}

In this second part we aim to apply the lessons learned in the smooth unconstrained case to obtain better guarantees on the broader class of composite problems. The results have a similar but not identical structure to those in the previous part and consequently we will \emph{redefine} many quantities to adapt them to their new roles.

\subsection{Composite minimization}

We consider the optimization problem
\begin{equation} \label{label_094}
\min_{x \in \mathbb{E}} F(x) = f(x) + \Psi(x),
\end{equation}
where $f$ is an $L_f$-smooth function with strong convexity parameter $\mu_f$. The regularizer $\Psi$ is an extended value convex lower-semicontinuous with strong convexity parameter $\mu_{\Psi}$ yielding an objective strong convexity parameter $\mu = \mu_f + \mu_{\Psi}$ for the entire objective. We assume that $\mu_f$ and $\mu_{\Psi}$ or at least lower estimates of each are known to the algorithm whereas $L_f$ may be unknown in which case the algorithm starts with a guess $L_0 > 0$ that can be arbitrary. The regularizer is proximable, meaning that the exact computation of the function
\begin{equation} \label{label_095}
\operatorname{prox}_{\tau \Psi}(x) \overset{\operatorname{def}}{=} \arg\min_{z \in \mathbb{E}}\left\{\tau \Psi(z) + \frac{1}{2} \| z - x \|^2 \right\},
\end{equation}
is tractable for all $x \in \mathbb{E}$ and $\tau > 0$. The regularizer is infinite outside the feasible set $Q$ being thus able to embed constraints.

\subsection{The bounds and estimate functions}
The oracle is called at the points $y_k$, $k \geq 1$. Under the composite setup we have $y_1 = x_0$, the initial point, and the points $x_k$, $k \geq 1$, are now given by
\begin{equation} \label{label_096}
x_k = T_{L_k}(y_k) = \operatorname{prox}_{\frac{1}{L_k} \Psi} \left(y_k - \frac{1}{L_k} B^{-1} f'(y_k)\right),
\end{equation}
where the proximal gradient operator $T_L$ is now defined for all $z \in \mathbb{E}$ and $L > 0$ as
\begin{equation}
T_L(z) \overset{\operatorname{def}}{=} \arg\min_{x \in \mathbb{E}} \left\{ f(z) + \langle f'(z), x - z \rangle + \frac{L}{2} \| x - z \|^2 + \Psi(x) \right\}.
\end{equation}
The parameter $L_k$ in \eqref{label_096} is an estimate of $L_f$ outputted by a line-search procedure equipped with a stopping criterion that ensures the descent rule
\begin{equation} \label{label_097}
f(x_k) \leq f(y_k) + \langle f'(y_k), x_k - y_k \rangle + \frac{L_k}{2} \| x_k - y_k \|^2, \quad k \geq 1.
\end{equation}
This setup can handle situations in which $L_f$ may not be known to the algorithm. Let $\bar{L}_k \overset{\operatorname{def}}{=} L_k + \mu_{\Psi}$ for all $k \geq 1$.

\begin{proposition} \label{label_098}
For any $L_k$ satisfying the descent rule in \eqref{label_097}, if $x_k$ is obtained from \eqref{label_096}, then we can construct a global lower bound on the objective of the form
\begin{equation}
F(x) \geq F(x_k) + \frac{1}{2 \bar{L}_k} \| g_k \|_*^2 + \langle g_k, x - y_k \rangle + \frac{\mu}{2} \| x - y_k \|^2, \quad x \in \mathbb{E}, \quad k \geq 1,
\end{equation}
where we redefine $g_k$ as the composite gradient mapping \cite{ref_015}, now given by
\begin{equation} \label{label_099}
g_k \overset{\operatorname{def}}{=} \bar{L}_k B (y_k - x_k), \quad k \geq 1.
\end{equation}
\end{proposition}
\begin{proof}
The first-order optimality conditions of \eqref{label_096} imply that for every $k \geq 1$ there exists a subgradient $\xi_{k}$ of $\Psi$ at $x_{k}$ (we use the notation $\xi_{k} \in \partial \Psi(x_{k})$) such that
\begin{equation} \label{label_100}
f'(y_k) + L_k B (x_k - y_k) + \xi_k = 0.
\end{equation}
Hereafter, we consider all $k \geq 1$ and $x \in \mathbb{E}$. The existence of strong convexity parameters $\mu_f \geq 0$ and $\mu_{\Psi} \geq 0$, respectively imply that
\begin{align}
f(x) &\geq f(y_k) + \langle f'(y_k), x - y_k \rangle + \frac{\mu_f }{2} \| x - y_k \|^2, \label{label_101}\\
\Psi(x) &\geq \Psi(x_k) + \langle \xi_k, x - x_k \rangle + \frac{\mu_\Psi}{2} \| x - x_k \|^2. \label{label_102}
\end{align}
Adding together \eqref{label_097}, \eqref{label_101} and \eqref{label_102} we obtain
\begin{equation}
\begin{gathered}
f(y_k) + \langle f'(y_k), x_k - y_k \rangle + \frac{L_k}{2} \| x_k - y_k \|^2 + f(x) + \Psi(x) \geq f(x_k) \\+ f(y_k) + \langle f'(y_k), x - y_k \rangle + \frac{\mu_f}{2} \| x - y_k \|^2 + \Psi(x_k) + \langle \xi_k, x - x_k \rangle + \frac{\mu_\Psi}{2} \| x - x_k \|^2. \label{label_103}
\end{gathered}
\end{equation}
Grouping together terms in \eqref{label_103} using \eqref{label_094} gives
\begin{equation} \label{label_104}
\begin{gathered}
F(x) \geq F(x_k) + \langle f'(y_k) + \xi_k, x - x_k \rangle \\- \frac{L_k}{2} \| x_k - y_k \|^2 + \frac{\mu_f}{2} \| x - y_k \|^2 + \frac{\mu_\Psi}{2} \| x - x_k \|^2.
\end{gathered}
\end{equation}
Applying \eqref{label_100} to \eqref{label_104} and re-basing everything around $y_k$ yields
\begin{equation} \label{label_105}
\begin{gathered}
F(x) \geq F(x_k) + L_k \langle B(y_k - x_k), x - y_k + y_k - x_k \rangle \\- \frac{L_k}{2} \| x_k - y_k \|^2 + \frac{\mu_f}{2} \| x - y_k \|^2 + \frac{\mu_{\Psi}}{2} \| (y_k - x_k) + (x - y_k)\|^2.
\end{gathered}
\end{equation}
Expanding the last square term and regrouping terms in \eqref{label_105} leads to
\begin{equation} \label{label_106}
\begin{gathered}
F(x) \geq F(x_k) + L_k \langle B(y_k - x_k), x - y_k \rangle + \frac{L_k + \mu_{\Psi}}{2} \| x_k - y_k \|^2 \\+ \frac{\mu}{2} \| x - y_k \|^2 + \mu_{\Psi} \langle B(y_k - x_k), x - y_k \rangle.
\end{gathered}
\end{equation}
Further regrouping terms in \eqref{label_106} using $\bar{L}_k = L_k + \mu_{\Psi}$ completes the proof. See also \cite{ref_010} for problems where $\| .\|$ denotes the standard Euclidean norm.
\end{proof}
Unlike in the smooth unconstrained case, the global lower bounds in Proposition~\ref{label_098} do not contain gradient-like quantities centered around the variable $x$ (this would correspond to variable $z_1$ in \eqref{label_007}), eliminating the need for a redundant first iteration. We can thus revert to the original indexing found in ACGM~\cite{ref_003} and instead rely on input parameters $A_0 \geq 0$ and $\gamma_0 > 0$. Adhering to this new indexing strategy, we define the bounds based on Proposition~\ref{label_098} and their aggregate for all $k \geq 1$ as
\begin{gather}
w_k(x) \overset{\operatorname{def}}{=} F(x_k) + \frac{1}{2 \bar{L}_k} \| g_k \|_*^2 + \langle g_k, x - y_k \rangle + \frac{\mu}{2} \| x - y_k \|^2, \label{label_107} \\
(A_k - A_0) W_k(x, g) \overset{\operatorname{def}}{=} \displaystyle \sum_{i = 1}^{k} a_i w_i(x), \label{label_108}
\end{gather}
with the convergence guarantees satisfying the positive difference property
\begin{equation} \label{label_109}
a_{k + 1} = A_{k + 1} - A_k, \quad k \geq 0.
\end{equation}
To construct bounds around the optimal point $x^*$, we introduce the bounds
\begin{equation}
\hat{w}_k(x) = \frac{\mu \alpha_k}{2} \| x - x_k \|^2, \quad k \geq 0.
\end{equation}
The existence of the strong convexity parameter $\mu$ implies the quadratic functional growth property~\cite{ref_013}. We thus have
\begin{equation} \label{label_107_star}
F(x_k) \geq F^* + \hat{w}_k(x^*) , \quad k \geq 0,\quad 0 \leq \alpha_k \leq 1.
\end{equation}
We introduce the dampening parameters $\alpha_k$ because we cannot prove the convergence of a method constructed with all $\alpha_k$ set to the maximal value of $1$, as we will show in the sequel.

Having defined the bounds, we can construct the estimate functions using the same structure as in \eqref{label_014}, namely
\begin{equation} \label{label_111}
\psi_k(x) = (A_k - A_0) W_k(x) + A_k \hat{w}_k(x) + A_0 (F(x_0) - \hat{w}_0(x)) + \frac{\gamma_0}{2} \|x - x_0 \|^2,
\end{equation}
for all $k \geq 0$. The estimate functions take on the canonical form
\begin{equation} \label{label_112}
\psi_k(x) = \psi_k^* + \frac{\gamma_k}{2} \| x - v_k \|^2, \quad k \geq 0,
\end{equation}
with the first estimate function having $\psi_0^* = A_0 F(x_0)$ and $v_0 = x_0$.

The Estimate Sequence Property is now centered around $x_k$, namely
\begin{equation} \label{label_018_comp}
A_k F(x_k) \leq \psi_k^*, \quad k \geq 0.
\end{equation}
Because in the general case the objective $F$ may not be differentiable, we can only guarantee the convergence of the iterates, which follows directly from the estimate sequence property in \eqref{label_018_comp} and our assumptions
\begin{equation} \label{label_114}
\begin{gathered}
A_k F(x_k) + \frac{\gamma_k}{2} \| v_k - x^* \|^2 \overset{\eqref{label_018_comp}}{\leq} \psi_k^* + \frac{\gamma_k}{2} \| v_k - x^* \|^2 \overset{\eqref{label_112}}{=} \psi_k(x^*) \overset{\eqref{label_111}}{\leq} A_k F^* + A_k \hat{w}_k(x^*) \\ + A_0 (F(x_0) - F^* - \hat{w}_0(x^*)) + \frac{\gamma_0}{2} \| x_0 - x^* \|^2 \overset{\eqref{label_107_star}}{\leq} A_k F(x_k) + \mathcal{D}_0, \quad k \geq 0,
\end{gathered}
\end{equation}
where $\mathcal{D}_0 \overset{\operatorname{def}}{=} A_0 (F(x_0) - F^* - \hat{w}_0(x)) + \frac{\gamma_0}{2} \| x_0 - x^* \|^2 \overset{\eqref{label_107_star}}{\geq} 0$.

Again, we have two possibilities of defining convergence guarantees. First, when $\mu = 0$, the estimate functions are identical to those used in deriving the Accelerated Composite Gradient Method~\cite{ref_003,ref_004,ref_006}. This situation is well studied and we will not give it more consideration in this work.

The second means is based on the distance between the estimate sequence optima and the optimal set whereby the chain in \eqref{label_114} implies that
\begin{equation} \label{label_115}
\| v_k - x^* \|^2 \leq \frac{2 \mathcal{D}_0}{\gamma_k}, \quad k \geq 0.
\end{equation}
Thus, we can enhance the performance in iterate space of first-order methods by increasing the growth rate of the estimate function curvature sequence $(\gamma_k)_{k \geq 0}$.

\subsection{Constructing an improved method}

We can combine the accumulated weight update in \eqref{label_108}, the canonical form in \eqref{label_112} and the bound in \eqref{label_107_star} to obtain a single update that can be used to derive the algorithm, given for all $k \geq 0$ by
\begin{equation} \label{label_116}
\begin{gathered}
\psi_{k + 1}^* + \frac{\gamma_{k + 1}}{2} \| x - v_{k + 1} \|^2 - \frac{\mu \alpha_{k + 1} A_{k + 1}}{2} \| x - x_{k + 1} \|^2 \\
= \psi_k^* + \frac{\gamma_k}{2} \| x - v_k \|^2 - \frac{\mu \alpha_{k} A_k}{2} \| x - x_k \|^2 \\
+ a_{k + 1} \left( F(x_{k + 1}) + \frac{1}{2 \bar{L}_{k + 1}} \| g_{k + 1} \|_*^2 + \langle g_{k + 1} , x - y_{k + 1} \rangle + \frac{\mu}{2} \| x - y_{k + 1} \|^2 \right).
\end{gathered}
\end{equation}
Differentiating \eqref{label_116} twice yields
\begin{equation} \label{label_117}
\gamma_{k + 1} = \gamma_k + \mu (a_{k + 1} + \alpha_{k + 1} A_{k + 1} - \alpha_k A_k), \quad k \geq 0.
\end{equation}

Differentiating \eqref{label_116} once and taking $x = y_{k + 1}$ provides an update of the estimate function optima for all $k \geq 0$ as
\begin{equation} \label{label_118}
\begin{gathered}
\gamma_{k + 1} (v_{k + 1} - y_{k + 1}) = \gamma_k (v_k - y_{k + 1}) - \mu \alpha_k A_k(x_k - y_{k + 1}) \\ - (a_{k + 1} + q_{k + 1} \alpha_{k + 1} A_{k + 1}) B^{-1} g_{k + 1},
\end{gathered}
\end{equation}
where the local inverse condition number is defined as $q_{k} \overset{\operatorname{def}}{=} \mu / (L_k + \mu_{\Psi})$, $k \geq 1$. To simplify the derivations we redefine $\bar{a}_{k} \overset{\operatorname{def}}{=} a_{k} + q_{k} \alpha_{k} A_{k}$, $k \geq 1$.

Let the estimate sequence gap in \eqref{label_018_comp} be written as $\Gamma_k \overset{\operatorname{def}}{=} \psi_k^* - A_k F(x_k)$, $k \geq 0$. Setting $x = y_{k + 1}$ in \eqref{label_116} and rearranging terms using \eqref{label_099} gives
\begin{equation} \label{label_119}
\begin{gathered}
\Gamma_{k + 1} - \Gamma_k = \frac{\gamma_k}{2} \| v_k - y_{k + 1} \|^2 - \frac{\gamma_{k + 1}}{2} \| v_{k + 1} - y_{k + 1} \|^2 - \frac{\mu \alpha_k A_k}{2} \| x_k - y_{k + 1} \|^2 \\ + \frac{\bar{a}_{k + 1}}{2 \bar{L}_{k + 1}} \| g_{k + 1} \|_*^2 + A_k(F(x_k) - F(x_{k + 1})), \quad k \geq 0.
\end{gathered}
\end{equation}
Applying Proposition~\ref{label_098} with $x = x_k$ to \eqref{label_119} yields for all $k \geq 0$
\begin{equation} \label{label_120}
\begin{gathered}
\Gamma_{k + 1} - \Gamma_k \geq \frac{\gamma_k}{2} \| v_k - y_{k + 1} \|^2 - \frac{\gamma_{k + 1}}{2} \| v_{k + 1} - y_{k + 1} \|^2 \\+ \frac{\mu (1 - \alpha_k) A_k}{2} \| x_k - y_{k + 1} \|^2 + \frac{(1 + q_{k + 1} \alpha_{k + 1}) A_{k + 1}}{2 \bar{L}_{k + 1}} \| g_{k + 1} \|_*^2 + A_k \langle g_{k + 1}, x_k - y_{k + 1} \rangle.
\end{gathered}
\end{equation}
\begin{theorem}
Any first-order method that updates the estimate functions according to \eqref{label_116} satisfies at every iteration
\begin{equation}
\begin{gathered}
\Gamma_{k + 1} - \Gamma_k \geq \frac{\gamma_k}{2 \bar{\gamma}^2_{k + 1}} (\bar{\gamma}_{k + 1}^2 - \gamma_k \gamma_{k + 1}) \| v_k - y_{k + 1} \|^2 \\
+ \left( \frac{(1 + q_{k + 1} \alpha_{k + 1}) A_{k + 1}}{2 \bar{L}_{k + 1}} - \frac{\bar{a}^2_{k + 1}}{2 \gamma_{k + 1}} \right) \| g_{k + 1} \|_*^2 + \frac{\mu (1 - \alpha_k) A_k}{2} \| x_k - y_{k + 1} \|^2 \\
+ \left\langle \frac{\bar{\gamma}_{k + 1}}{\gamma_{k + 1}} g_{k + 1} + \mu \alpha_{k} B \left( \frac{\gamma_k}{\bar{\gamma}_{k + 1}} (v_k - y_{k + 1}) - \frac{\mu \alpha_k}{2 \gamma_{k + 1}} Y_{k + 1} \right), Y_{k + 1} \right\rangle , \quad k \geq 0,
\end{gathered}
\end{equation}
where
\begin{align}
\bar{\gamma}_{k + 1} & \overset{\operatorname{def}}{=} \gamma_{k + 1} - \mu \alpha_k \bar{a}_{k + 1}, \label{label_121}\\
Y_{k + 1} & \overset{\operatorname{def}}{=} A_k (x_k - y_{k + 1}) + \frac{\bar{a}_{k + 1} \gamma_k}{\bar{\gamma}_{k + 1}}(v_k - y_{k + 1}). \label{label_122}
\end{align}
\end{theorem}
\begin{proof}
The results herein hold for all $k \geq 0$. The proof structure follows closely the one of Theorem~\ref{label_033}. First, using \eqref{label_118} and \eqref{label_122} we obtain
\begin{equation} \label{label_123}
\gamma_{k + 1} (v_{k + 1} - y_{k + 1}) = \frac{\gamma_k \gamma_{k + 1}}{\bar{\gamma}_{k + 1}}(v_k - y_{k + 1}) - \mu \alpha_k Y_{k + 1} - \bar{a}_{k + 1} B^{-1} g_{k + 1}.
\end{equation}
Consequently
\begin{equation} \label{label_124}
\begin{gathered}
\frac{\gamma_{k}}{2} \| v_{k} - y_{k + 1} \|^2 - \frac{\gamma_{k + 1}}{2} \| v_{k + 1} - y_{k + 1} \|^2 \\= \frac{\gamma_{k}}{2} \| v_{k} - y_{k + 1} \|^2 - \frac{1}{2 \gamma_{k + 1}} \left\| \frac{\gamma_k \gamma_{k + 1}}{\bar{\gamma}_{k + 1}}(v_k - y_{k + 1}) - \mu \alpha_k Y_{k + 1} - \bar{a}_{k + 1} B^{-1} g_{k + 1} \right\|^2 \\
= \frac{\gamma_k}{2 \bar{\gamma}_{k + 1}^2} \left( \bar{\gamma}_{k + 1}^2 - \gamma_k \gamma_{k + 1} \right) \| v_k - y_{k + 1} \|^2 - \frac{\mu^2 \alpha_k^2}{2 \gamma_{k + 1}} \| Y_{k + 1} \|^2 - \frac{\bar{a}_{k + 1}^2}{2 \gamma_{k + 1}} \| g_{k + 1} \|_*^2 \\+ \frac{\bar{a}_{k + 1} \gamma_k}{\bar{\gamma}_{k + 1}} \langle g_{k + 1}, v_k - y_{k + 1} \rangle + \mu \alpha_k \left\langle \frac{\gamma_k}{\bar{\gamma}_{k + 1}} B (v_k - y_{k + 1}) - \frac{\bar{a}_{k + 1}}{\gamma_{k + 1}} g_{k + 1} , Y_{k + 1} \right\rangle .
\end{gathered}
\end{equation}
Expanding \eqref{label_120} using \eqref{label_124} and further applying \eqref{label_122} gives
\begin{equation} \label{label_125}
\begin{gathered}
\Gamma_{k + 1} - \Gamma_k \geq \frac{\gamma_k}{2 \bar{\gamma}_{k + 1}^2} \left( \bar{\gamma}_{k + 1}^2 - \gamma_k \gamma_{k + 1} \right) \| v_k - y_{k + 1} \|^2
\\ + \left( \frac{(1 + q_{k + 1} \alpha_{k + 1}) A_{k + 1}}{2 \bar{L}_{k + 1}} - \frac{\bar{a}_{k + 1}^2}{2 \gamma_{k + 1}} \right)
\| g_{k + 1} \|_*^2 + \frac{\mu (1 - \alpha_k) A_k}{2} \| x_k - y_{k + 1} \|^2
\\+ \left\langle \left(1 - \frac{\mu \bar{a}_{k + 1}}{\gamma_{k + 1}} \right) g_{k + 1} + \mu B \left( \frac{\gamma_k}{\bar{\gamma}_{k + 1}} (v_k - y_{k + 1}) - \frac{\mu}{2 \gamma_{k + 1}} Y_{k + 1} \right), Y_{k + 1} \right\rangle.
\end{gathered}
\end{equation}
We conclude by applying \eqref{label_121} in \eqref{label_125}.
\end{proof}
The estimate sequence property can be maintained across iterations regardless of the algorithmic state if we can simultaneously ensure for all $k \geq 0$ three conditions:
\begin{align}
Y_{k + 1} &= 0, \label{label_126}\\
(1 + q_{k + 1} \alpha_{k + 1}) A_{k + 1} \gamma_{k + 1} &\geq \bar{L}_{k + 1} \bar{a}_{k + 1}^2, \label{label_127}\\
\bar{\gamma}_{k + 1}^2 &\geq \gamma_k \gamma_{k + 1} . \label{label_128}
\end{align}
The first condition \eqref{label_126} gives the update of the oracle point $y_{k + 1}$ as
\begin{equation} \label{label_129}
y_{k + 1} = (A_k \bar{\gamma}_{k + 1} + \bar{a}_{k + 1} \gamma_k)^{-1} (A_k \bar{\gamma}_{k + 1} x_k + \bar{a}_{k + 1} \gamma_k v_k), \quad k \geq 0,
\end{equation}
where $\bar{a}_{k + 1}$ and $\bar{\gamma}_{k + 1}$ can be expressed as functions of the quantities $A_k$ and $\gamma_k$ available at the beginning of the iteration alongside $L_{k + 1}$ (with the closely related $q_{k + 1}$) and $a_{k + 1}$ generated during the iteration for all $k \geq 0$ as
\begin{align}
\bar{a}_{k + 1} &= a_{k + 1} + q_{k + q} \alpha_{k + 1} A_{k + 1} = q_{k + 1} \alpha_{k + 1} A_k + (1 + q_{k + 1} \alpha_{k + 1}) a_{k + 1}, \label{label_130}\\
\bar{\gamma}_{k + 1} &\overset{\eqref{label_121}}{=} \gamma_k + \mu a_{k + 1} + \mu \left( \alpha_{k + 1} - \alpha_k - q_{k + 1} \alpha_k \alpha_{k + 1}\right) (A_k + a_{k + 1}).
\end{align}
From \eqref{label_123} and \eqref{label_126} we also obtain the estimate sequence optima update
\begin{equation} \label{label_131}
v_{k + 1} = \frac{\gamma_k}{\bar{\gamma}_{k + 1}} v_k + \left( 1 - \frac{\gamma_k}{\bar{\gamma}_{k + 1}} \right) y_{k + 1} - \frac{\bar{a}_{k + 1}}{\gamma_{k + 1}} B^{-1} g_{k + 1} .
\end{equation}

When $\mu = 0$, the condition \eqref{label_128} is always satisfied with equality and \eqref{label_127} becomes $L_{k + 1} a_{k + 1}^2 \leq \gamma_0 A_{k + 1}$, $k \geq 0$, which is identical to the non-strongly convex version of ACGM~\cite{ref_008,ref_009}.

When $\mu > 0$, \eqref{label_127} and \eqref{label_128} imply that $a_{k + 1}^{(2)} \overset{\eqref{label_128}}{\leq} a_{k + 1} \overset{\eqref{label_127}}{\leq} a_{k + 1}^{(1)}$, where $a_{k + 1}^{(1)}$ and $a_{k + 1}^{(2)}$ are, respectively, the only positive solutions of the following quadratic equations:
\begin{equation} \label{label_132}
(\bar{L}_{k + 1} - \mu) \left( a_{k + 1}^{(1)} \right)^2 - (\gamma_k + \mu (1 - \alpha_{k}) A_k) a_{k + 1}^{(1)} - A_k \left( \gamma_k + \frac{\mu \beta_{k + 1} A_k}{1 + q_{k + 1}} \right) = 0,
\end{equation}
\begin{equation} \label{label_133}
\begin{gathered}
\mu (1 + \beta_{k + 1})^2 \left( a_{k + 1}^{(2)} \right)^2 + \left( (1 - \alpha_{k + 1} + 2 \beta_{k + 1}) \gamma_k + 2 \mu \beta_{k + 1} (1 + \beta_{k + 1}) A_k \right) a_{k + 1}^{(2)} \\
+ A_k ( (\alpha_k - \alpha_{k + 1} + 2 \beta_{k + 1}) \gamma_k + \mu \beta_{k + 1}^2 A_k) = 0,
\end{gathered}
\end{equation}
with $\beta_{k + 1} \overset{\operatorname{def}}{=} \alpha_{k + 1} - \alpha_{k} - q_{k + 1} \alpha_{k} \alpha_{k + 1}$, $k \geq 0$.

A sufficient condition for the algorithm to maintain the estimate sequence property is thus given by
\begin{equation} \label{label_134}
a_{k + 1}^{(2)} \leq a_{k + 1}^{(1)}, \quad k \geq 0.
\end{equation}
When \eqref{label_134} holds, the most aggressive weight update would be $a_{k + 1} = a_{k + 1}^{(1)}$, $k \geq 0$.

We are now ready to formulate a method. Putting together the new iterate generation in \eqref{label_096} under the line-search condition in \eqref{label_097}, the weight update in \eqref{label_132}, accumulating the convergence guarantees as in \eqref{label_109}, the estimate function curvature update in \eqref{label_117} and the auxiliary quantity updates in \eqref{label_130} and \eqref{label_121}, respectively, enabling the auxiliary point update in \eqref{label_129} along with the estimate function optima update in \eqref{label_131} we obtain Algorithm~\ref{label_135}. Note that Algorithm~\ref{label_135} leaves the choice of dampening parameters $\alpha_k \geq 0$ to the user. The parameter list also includes the lower bound on the Lipschitz estimate $L_l$, that will be required by certain approaches for updating the dampening parameters described in the sequel.

\begin{algorithm}[h!]
\caption{An Enhanced Accelerated Composite Gradient Method}
\label{label_135}
\begin{algorithmic}[1]
\STATE {\bf Input:} $B \succ 0$, ${x}_0 \in \mathbb{R}^n$, $\mu_f \geq 0$, $\mu_{\Psi} \geq 0 $, $\alpha_0 \in [0, 1]$, \\$L_0 > 0$, $L_l \geq 0$, $r_u > 1, r_d \in (0, 1]$, $T \inn{\infty}$
\STATE ${v}_0 = {x}_0$, $\mu = \mu_f + \mu_{\Psi}$
\STATE $A_0 = 0$, $\gamma_0 = 1$
\FOR{$k = 0,\ldots{},T-1$}
\STATE Select $\alpha_{k + 1}$ to ensure that $\Gamma_{k + 1} \geq 0$
\STATE $\tilde{\gamma}_k = \gamma_k + \mu (1 - \alpha_k) A_k$
\STATE $L_{k + 1} := \max\{L_l, r_d L_k\}$
\LOOP
\STATE $q_{k + 1} := \frac{\mu}{L_{k + 1} + \mu_{\Psi}}$
\STATE $\bar{\beta}_{k + 1} = \frac{\alpha_{k + 1}}{1 + q_{k + 1} \alpha_{k + 1}} - \alpha_k$
\STATE $a_{k + 1} = \frac{1}{2 (\bar{L}_{k + 1} - \mu)} \left( \tilde{\gamma}_k + \sqrt{\tilde{\gamma}_k^2 + 4 (\bar{L}_{k + 1} - \mu) A_k (\gamma_k + \mu \bar{\beta}_{k + 1} A_k) } \right)$ \label{label_136}
\STATE $A_{k + 1} = A_k + a_{k + 1}$
\STATE $\bar{a}_{k + 1} = a_{k + 1} + q_{k + 1} \alpha_{k + 1} A_{k + 1}$
\STATE $\gamma_{k + 1} = \gamma_k + \mu \left(a_{k + 1} + \alpha_{k + 1} A_{k + 1} - \alpha_k A_k\right)$
\STATE $\bar{\gamma}_{k + 1} = \gamma_{k + 1} - \mu \alpha_k \bar{a}_{k + 1}$
\STATE $y_{k + 1} = \left( A_k \bar{\gamma}_{k + 1} + \bar{a}_{k + 1} \gamma_k \right)^{-1} \left( A_k \bar{\gamma}_{k + 1} x_k + \bar{a}_{k + 1} \gamma_k v_k \right)$ \label{label_137}
\STATE ${x}_{k + 1} := T_{L_{k + 1}}({y}_{k + 1})$ \label{label_138}
\IF {$f({x}_{k + 1}) \leq f({y}_{k + 1}) + \langle f'({y}_{k + 1}), {x}_{k + 1} - {y}_{k + 1} \rangle + \frac{L_{k + 1}}{2} \| {x}_{k + 1} - {y}_{k + 1} \|_2^2$} \label{label_139}
\STATE Break from loop
\ELSE
\STATE $L_{k + 1} := r_u L_{k + 1}$
\ENDIF
\ENDLOOP
\STATE $v_{k + 1} = \frac{\gamma_k}{\bar{\gamma}_{k + 1}} v_k + \left( 1 - \frac{\gamma_k}{\bar{\gamma}_{k + 1}} \right) y_{k + 1} - \frac{\bar{a}_{k + 1}}{\gamma_{k + 1}} B^{-1} g_{k + 1}$
\ENDFOR
\end{algorithmic}
\end{algorithm}

\subsection{Choosing the dampening parameters}

The challenge lies in selecting a suitable strategy for computing the dampening parameters $\alpha_k$, $k \geq 0$ to uphold \eqref{label_134}. In the following, we present two very simple but pessimistic approaches.

\begin{proposition} \label{label_140}
Let $L_l$ be the lowest Lipschitz estimate encountered by the algorithm satisfying the descent rule \eqref{label_097} and let $q_l \overset{\operatorname{def}}{=} \mu / (L_l + \mu_{\Psi})$. When $\alpha_k = \alpha$ for all $k \geq 0$, $\gamma_0 \geq \mu (1 + \alpha) A_0$, $q_l \leq 1/3$ and $0 \leq \alpha \leq \alpha_{\operatorname{max}}(q_l)$ then \eqref{label_127} can be satisfied with equality without invalidating \eqref{label_128} for all $k \geq 0$, regardless of the algorithmic state. Here $\alpha_{\operatorname{max}}(q)$ satisfies the following equation:
\begin{equation} \label{label_141}
\alpha_{\operatorname{max}}(q) = \alpha \mbox{ s.t. } \delta(q, \alpha) = 0,
\end{equation}
where we define the function $\delta$ for any $q \in [0, 1]$ and $\alpha \in [0, 1]$ as
\begin{equation}
\delta(q, \alpha) \overset{\operatorname{def}}{=} (1 - \alpha) \sqrt{(1 + \alpha) (1 + q \alpha)} - \sqrt{q} \alpha \left(1 - q \alpha^2 \right).
\end{equation}
\end{proposition}
\begin{proof}
Within this proof we consider all $k \geq 0$ unless specified otherwise.

The assumption $\alpha_k = \alpha$ gives a growth rate of the estimate function curvature as
\begin{equation}
\gamma_{k + 1} = \gamma_k + \mu (1 + \alpha) a_{k + 1} = \gamma_0 - \mu (1 + \alpha) A_0 + \mu(1 + \alpha) A_{k + 1}.
\end{equation}
Because $\gamma_0 \geq \mu (1 + \alpha) A_0$ it follows that $\frac{\gamma_{k + 1}}{A_{k + 1}} \geq \mu (1 + \alpha)$. Equality in \eqref{label_127} leads to
\begin{equation} \label{label_142}
\left(\frac{\bar{a}_{k + 1}}{A_{k + 1}}\right)^2 = \frac{(1 + q_{k + 1} \alpha) \gamma_{k + 1}}{\bar{L}_{k + 1} A_{k + 1}} \geq q_{k + 1} (1 + \alpha) (1 + q_{k + 1} \alpha).
\end{equation}

From the definition of $\bar{a}_k$ we also have
\begin{equation} \label{label_143}
\frac{\bar{a}_{k + 1}}{A_{k + 1}} = \frac{a_{k + 1}}{A_{k + 1}} + q_{k + 1} \alpha.
\end{equation}

The gap in \eqref{label_128} can be expressed as
\begin{equation} \label{label_144}
\begin{gathered}
\bar{\gamma}_{k + 1}^2 - \gamma_k \gamma_{k + 1} = (\gamma_{k + 1} - \mu \alpha \bar{a}_{k + 1})^2 - \gamma_{k + 1} (\gamma_{k + 1} - \mu (1 + \alpha) a_{k + 1}) \\
= \mu A_{k + 1} \left( \mu \alpha^2 A_{k + 1} \left(\frac{\bar{a}_{k + 1}}{A_{k + 1}}\right)^2 + \gamma_{k + 1} \left( - 2 \alpha \frac{\bar{a}_{k + 1}}{A_{k + 1}} + (1 + \alpha) \frac{a_{k + 1}}{A_{k + 1}} \right) \right) \\
= \mu A_{k + 1} \gamma_{k + 1} \left( q_{k + 1} \alpha^2(1 + q_{k + 1} \alpha) + (1 + \alpha) \left( \frac{\bar{a}_{k + 1}}{A_{k + 1}} - q_{k + 1} \alpha \right) - 2 \alpha \frac{\bar{a}_{k + 1}}{A_{k + 1}} \right) \\
= \mu A_{k + 1} \gamma_{k + 1} \left( (1 - \alpha) \frac{\bar{a}_{k + 1}}{A_{k + 1}} - q_{k + 1} \alpha \left( 1 - q_{k + 1} \alpha^2 \right) \right).
\end{gathered}
\end{equation}
Combining \eqref{label_142} and \eqref{label_144} gives
\begin{equation}
\bar{\gamma}_{k + 1}^2 - \gamma_k \gamma_{k + 1} \geq \mu A_{k + 1} \gamma_{k + 1} \sqrt{q_{k + 1}} \delta(q_{k + 1}, \alpha).
\end{equation}
The function $\delta(q, \alpha)$ is decreasing in both $q$ and $\alpha$ in the range $q \in [0, 1/3]$, $\alpha \in [0, 1]$.
\end{proof}
The sixth degree equation in \eqref{label_141} may not be solvable by radicals. However, the monotonicity of the function $\delta(q, \alpha)$ in $\alpha \in [0, 1]$ for any $q \in [0, 1]$ implies that \eqref{label_141} can be solved up to an arbitrary accuracy, such as machine precision, using bisecting search with negligible computational cost. Note that from our assumptions on the objective, we always have $q_k \leq 1$, $k \geq 1$, regardless of the algorithmic state, therefore the range of $q$ should never exceed $[0,1]$.

The function $\alpha_{\operatorname{max}}(q)$ is strictly convex in $q$ and can be shown numerically to attain its minimum over $q \in [0, 1]$ at $\alpha_{\operatorname{max}}^* \approx \alpha_{\operatorname{max}}(0.4733) \approx 0.7542$. The approximate values have been rounded down. Therefore, when a best case Lipschitz estimate cannot be established beforehand, we can provide the following worst-case choice of $\alpha$ for a value of $L_l = 0$.

\begin{corollary} \label{label_145}
When $\alpha_k = 0.7542$, $k \geq 0$, we have $\delta(q, \alpha) \geq 0$ for all $q \in [0, 1]$. Consequently, equality in \eqref{label_127} always ensures \eqref{label_128} without the need to impose any constraints on the local condition number.
\end{corollary}

\subsection{Convergence analysis}
Having established sufficient conditions for convergence, it remains to determine a bound on the worst-case rate.
\begin{proposition} \label{label_146}
Let $L_u \overset{\operatorname{def}}{=} \max\{r_d L_0, r_u L_f\}$ be the worst-case Lipschitz estimate and let $q_u \overset{\operatorname{def}}{=} \mu / (L_u + \mu_{\Psi})$. When $\alpha_k = \alpha$, $k \geq 0$ and $\gamma_0 \geq \mu (1 + \alpha) A_0$, the convergence guarantee increase is governed by
\begin{equation}
A_{k + 1} \geq \left( 1 - \sqrt{q_{k + 1}} r(q_{k + 1}, \alpha) \right)^{-1} A_k \geq \left( 1 - \sqrt{q_u} r(q_u, \alpha) \right)^{-1} A_k, \quad k \geq 0,
\end{equation}
where the ratio $r(q, \alpha)$ is given for any $q \in [0, 1]$ and $\alpha \in [0, 1]$ by
\begin{equation}
r(q, \alpha) \overset{\operatorname{def}}{=} \sqrt{(1 + \alpha) (1 + q \alpha)} - \sqrt{q} \alpha.
\end{equation}
\end{proposition}
\begin{proof}
Combining \eqref{label_142} and \eqref{label_143} we obtain for all $k \geq 0$ that
\begin{equation} \label{label_147}
\left(\left(1 - \frac{A_k}{A_{k + 1}}\right) - q_{k + 1} \alpha \right)^2 \geq q_{k + 1} (1 + \alpha) (1 + q_{k + 1} \alpha).
\end{equation}
Taking the square root and rearranging terms in \eqref{label_147} gives the first inequality. Noting that $\sqrt{q} r(q, \alpha)$ is an increasing function in $q$ over the entire range $q \in [0, 1]$ for any $\alpha \in [0, 1]$ and that $q_{k + 1} \leq q_u$ for all $k \geq 0$, we obtain the second inequality.
\end{proof}
When $A_0 = 0$, we have $\bar{\gamma}_0 = \gamma_0 = 1$ and consequently $A_1 = a_1 = \frac{1}{L_1 - \mu_f} \geq \frac{1}{L_u - \mu_f}$. If we also have $\alpha_k = \alpha$, $k \geq 0$, the conditions in Proposition~\ref{label_146} are met and hence we have
\begin{equation} \label{label_148}
A_k \geq (1 - r(q_u, \alpha) \sqrt{q_u})^{1 - k} \frac{1}{L_u - \mu_f}, \quad k \geq 1.
\end{equation}
Provided that the estimate sequence property is maintained at runtime by Algorithm~\ref{label_135}, such as when $\alpha$ is selected according to Proposition~\ref{label_140} or Corollary~\ref{label_145}, we obtain the following worst-case rate for all $k \geq 1$ on the iterates
\begin{equation} \label{label_149}
\| v_k - x^* \|^2 \overset{\eqref{label_115}}{\leq} \frac{2 \mathcal{D}_0}{\gamma_k} \leq \frac{2 \mathcal{D}_0}{\mu (1 + \alpha) A_k} \overset{\eqref{label_148}}{\leq} \frac{L_u - \mu_f}{\mu (1 + \alpha)} \left( 1 - r(q_u, \alpha) \sqrt{q_u} \right)^{k - 1} \| x_0 - x^* \|^2 .
\end{equation}

When we compare \eqref{label_149} with the convergence rate for estimate function optima in ACGM~\cite{ref_010,ref_003}, given by
\begin{equation} \label{label_150}
\| v_k - x^* \|^2 \leq \frac{L_u - \mu_f}{\mu} \left( 1 - \sqrt{q_u} \right)^{k - 1} \| x_0 - x^* \|^2,
\end{equation}
we see that the difference in the asymptotic rate is given by the term $r(q_u, \alpha)$. We list in Table~\ref{label_151} the values of this term, computed using either the value $\alpha_{\operatorname{max}}(q_l)$ recommended by Proposition~\ref{label_140} or using the ideal case of $\alpha = 1$. We see clearly that the worst-case rates given by Proposition~\ref{label_140} are almost indistinguishable from the ideal case. Table~\ref{label_151} also confirms that the values of $r(q_u, \alpha)$ lie within the range $[1, \sqrt{2}]$, approaching $\sqrt{2}$ for very small values of $q_u$, the situation most commonly encountered in practice.

\begin{table}[h]
\caption{Values of $\alpha_{\operatorname{max}}$ and the ratio $r$ for various combinations of $q_l$ and $q_u$ as well as in the ideal case}
\label{label_151}
\setlength{\tabcolsep}{4.5pt}
\centering \footnotesize
\begin{tabular}{lcccccccccc} \toprule
$q_l$ & $10^{-7}$ & $10^{-6}$ & $10^{-5}$ & $10^{-4}$ & $10^{-3}$ & $10^{-2}$ & $10^{-1}$ & $1/3$ & $0.4733$ & 1 \\
$\alpha_{\operatorname{max}}(q_l)$ & 0.9998 & 0.9993 & 0.9978 & 0.9930 & 0.9780 & 0.9337 & 0.8268 & 0.7614 & 0.7542 & 1.0000 \\
$r(q_l, 1)$ & 1.4139 & 1.4132 & 1.4111 & 1.4043 & 1.3833 & 1.3213 & 1.1670 & 1.0556 & 1.0286 & 1.0000 \\ \midrule
$q_u / q_l$ & \multicolumn{10}{c}{$r(q_u, \alpha_{\operatorname{max}}(q_l))$} \\ \midrule
$1$ & 1.4138 & 1.4130 & 1.4103 & 1.4019 & 1.3762 & 1.3037 & 1.1449 & 1.0465 & 1.0240 & 1.0000 \\
$10^{- 1}$ & 1.4140 & 1.4136 & 1.4124 & 1.4086 & 1.3967 & 1.3617 & 1.2745 & 1.2049 & 1.1838 & 1.1670 \\
$10^{- 2}$ & 1.4141 & 1.4139 & 1.4131 & 1.4107 & 1.4033 & 1.3813 & 1.3260 & 1.2849 & 1.2749 & 1.3213 \\
$10^{- 3}$ & 1.4141 & 1.4139 & 1.4133 & 1.4114 & 1.4055 & 1.3876 & 1.3434 & 1.3134 & 1.3083 & 1.3833 \\
$10^{- 4}$ & 1.4141 & 1.4140 & 1.4134 & 1.4116 & 1.4061 & 1.3897 & 1.3490 & 1.3228 & 1.3193 & 1.4043 \\
$10^{- 5}$ & 1.4141 & 1.4140 & 1.4134 & 1.4117 & 1.4063 & 1.3903 & 1.3508 & 1.3258 & 1.3228 & 1.4111 \\ \bottomrule
\end{tabular}
\end{table}

\clearpage
\section{Simulations}

Our two methods operate within two progressively larger problem classes: smooth unconstrained and composite. We select two synthetic problem instances representative of each class to test the effectiveness of each method.

\subsection{Smooth unconstrained problems}

The benchmark for this class consists of a Smoothed Piece-wise Linear objective problem (SPL) and an ill-conditioned quadratic problem (QUAD), both regularized with a quadratic term in order to add known strong convexity to the objective. The oracle functions and their structure are listed in Table~\ref{label_152}.

\begin{table}[h]
\centering
\small
\caption{Smooth unconstrained test problem oracle functions and their structure}
\label{label_152}
\begin{tabular}{ccc} \toprule
& SPL & QUAD~1 \\ \midrule
$ f(x)$ & $s \mathcal{E}\left(\frac{1}{s}(A x - b) \right) + \frac{\mu}{2} \|x \|^2$ & $\frac{1}{2} \langle x, A x \rangle + \frac{\mu}{2} \|x \|^2$ \\[2 mm]
$ f'(x)$ & $A^T \mathcal{S} \left(\frac{1}{s}(A x - b) \right) + \mu x $ & $A x + \mu x$ \\[2 mm]
$A$ & $\hat{A} - 1_M \left( \mathcal{S}\left( -\frac{1}{s} b \right)^T \hat{A} \right)$ & $\operatorname{diag}(\sigma)$, $\sigma_i = i/n$ \\[2 mm]
$x_0$ & uniformly random on $[-1,1]$, normalized to have $\|x_0\|=1$ & $x_0^{(i)} = 1/ \sigma_i$ \\ \bottomrule
\end{tabular}
\end{table}

In SPL, $\hat{A}$ is an $M = 2400 \times n = 400$ matrix and $b$ is a vector of $M$ variables. All entries of $\hat{A}$ and $b$, as well as of the starting point $x_0$, were sampled independently and uniformly from the interval $[-1, 1]$. The point $x_0$ was later normalized to lie on the unit sphere. For normalization, and throughout all our simulations, we have used the standard Euclidean norm. We chose the smoothing parameter to be $s = 0.05$. The logsumexp $\mathcal{E}$ and softmax $\mathcal{S}$ are, respectively, defined as
\begin{equation}
\mathcal{E}(z) \overset{\operatorname{def}}{=} \log \left( \displaystyle \sum_{i = 1}^{M} e^{z_i} \right) , \quad
\mathcal{S}(z)_i \overset{\operatorname{def}}{=} \left( \displaystyle \sum_{j = 1}^{M} e^{z_j} \right)^{-1} e^{z_i}, \quad z \in \mathbb{R}^M, \ i \inn{M} .
\end{equation}
Without regularization, the gradient of the objective would have a global Lipschitz constant of $L_{f_{(\mu)}} = \frac{1}{s} \displaystyle \max_{i \inn{n}}{ \|a_i\|}$, where $a_i$, $i \inn{n}$, are the columns of $A$.

In QUAD, $A$ is a diagonal matrix of size $n = 1000$ with the diagonal entries, corresponding to the eigenvalues, given by $\sigma_i = i/n$ for $i \inn{n}$. Note that this objective exhibits the strong convexity property with parameter $\bar{\mu} = 1/n = 10^{-3}$. However, to study the behavior of our method in the presence of \emph{unknown} strong convexity, we do not supply this information to any of the methods and instead we separately employ regularization.

For both problems, we choose $\mu = 10^{-4} L_{f_{(\mu)}}$, increasing the smoothness constant of the objective to $L_f = L_{f_{(\mu)}} + \mu = 1.0001 L_{f_{(\mu)}}$. The optimal point for both problems is thus unique and given by $0_n$, the vector of all zeros.

Our method encompasses both ITEM and TMM. To study the effect of adding memory to each of these particular cases, we consider OGMM with two distinct parameter choices. For $A_1 = 0$ and $\gamma_1 = 1$, OGMM constitutes ITEM with Memory (ITEM-M) while for $A_1 = 1$ and $\gamma_1 = 2 \mu r$, OGMM becomes TMM with Memory (TMM-M). To showcase the effectiveness of our adaptive mechanism, we compare our two instances of OGMM with the original ITEM and TMM as well as with the ACGM equipped with a fully-adaptive line-search procedure whereby the local Lipschitz estimate is allowed to decrease.

All methods are supplied with the correct values of $\mu$ and $L_f$. However, the value of $\bar{\mu}$ in QUAD is kept hidden from all algorithms. The bundle size in OGMM is set to $8$. This value is recommended by the good performance on smooth problems displayed by the previously introduced gradient methods with memory (see \cite{ref_016,ref_005,ref_007}). The bundle replacement strategy for OGMM is cyclic, wherein the newest entries displace the oldest. The middle method in OGMM is limited to $N = 2$ Newton iterations per outer iteration. The inner scheme of GMM is a version of the Fast Gradient Method optimized to operate on the simplex, previously employed by other types of gradient methods with memory (e.g., in \cite{ref_005,ref_006,ref_007}). The inner scheme terminates when either $T = 100$ iterations are expended or an absolute inner objective accuracy of $\delta = 10^{-12}$ is reached, whichever occurs first.
ACGM uses $A_0 = 0$ and $\gamma_0 = 1$ along with line-search parameters $r_u = 2$ and $r_d = 0.9$, according to the original algorithm specification in \cite{ref_009,ref_010}.

As mentioned previously, OGMM is able to efficiently and adaptively converge in iterate space, and throughout our simulations, we focus on this type of convergence behavior. We stop all methods at iteration $k$ when the relative error in iterate space $\epsilon_{rel} \overset{\operatorname{def}}{=} \|v_k - x^*\| / \|x_0 - x^*\|$ drops below $10^{-5}$. To highlight the convergence guarantee increase, as well as to display decay rates that resemble those of the function residual, we actually plot $\|v_k - x^*\|^2$ versus iterations. The convergence rates of the tested methods on both problems are shown in Figure~\ref{label_153}

\begin{figure}[h] \centering \footnotesize
\begin{minipage}[t]{0.48\linewidth} \centering
\includegraphics[width=\textwidth]{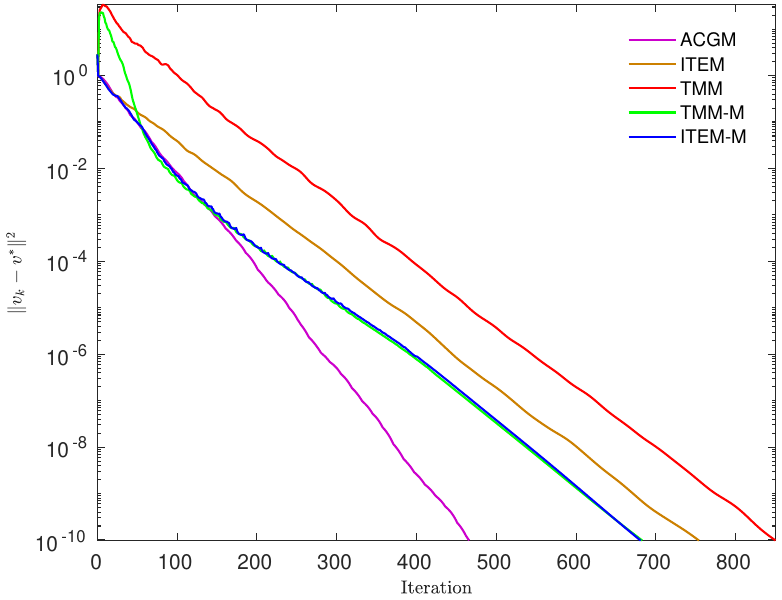}
(a) Convergence rates on SPL
\end{minipage}
\hspace{3mm}
\begin{minipage}[t]{0.48\linewidth} \centering
\includegraphics[width=\textwidth]{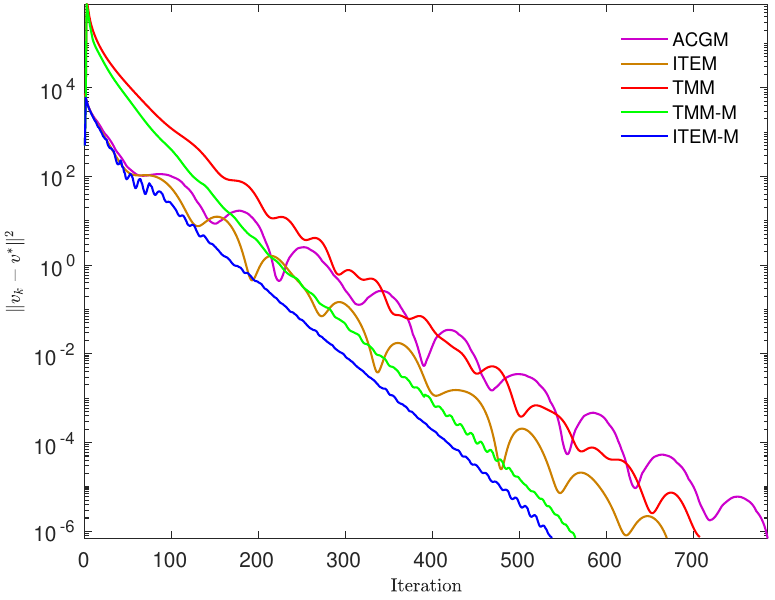}
(b) Convergence rates on QUAD
\end{minipage}
\caption{Convergence behavior of the OGMM instances ITEM-M and TMM-M with bundle size $m = 8$ as well as the original ITEM, TMM and ACGM, all measured in iterate space as $\|v_k - x^*\|^2$ vs. iterations on the smooth unconstrained problem benchmark}
\label{label_153}
\end{figure}

On both problems, ITEM converges significantly faster than TMM. The main reason appears to be an initial divergence allowed by the convergence analysis when $A_1 > 0$. This initial discrepancy is preserved as the algorithms progress, explaining the performance gap. Interestingly, TMM-M also exhibits the same divergence initially, but the adaptive mechanism manages to gradually correct the initial discrepancy. However, the performance of TMM-M still does not exceed that of ITEM-M in iterations on either problem. While the two instances OGMM show similar performance, they both consistently surpass their memory-less counterparts. ACGM is able to take advantage of SPL structure, wherein the local values of the Lipschitz constant are much lower than the global one. However, on the ill-conditioned QUAD, ACGM lags behind all other methods, even the ones without an adaptive mechanism.

On QUAD, none of the methods seem to take full advantage of the hidden parameter $\bar{\mu}$. All methods without memory display the oscillating behavior that is typical in this situation whereas both OGMM instances do not. This may be one means in which the convergence guarantee adjustment procedure can deal with hidden strong convexity. Regardless of the actual mechanics, OGMM displays good and surprisingly stable performance over the entire benchmark.

\subsection{Composite problems} We test the capabilities of the Enhanced ACGM on two simple synthetic composite problems: a dense Elastic Net~\cite{ref_020} problem and an Elastic Net regularized sparse Logistic Regression problem. The oracle functions pertaining to the objective in both problems are described in Table~\ref{label_154} using the sum softplus function $\mathcal{I}(z)$, the element-wise logistic function $\mathcal{L}(z)$ and the shrinkage operator $\mathcal{T}_{\tau}(x)$, respectively given by
\begin{gather}
\mathcal{I}(z) \overset{\operatorname{def}}{=} \displaystyle \sum_{i = 1}^{M} \log(1 + e^{z_i}), \quad \mathcal{L}(z)_i \overset{\operatorname{def}}{=} \frac{1}{1 + e^{-z_i}}, \quad z \in \mathbb{R}^M, \quad i \inn{M} \\
\mathcal{T}_{\tau}(x)_j \overset{\operatorname{def}}{=} \left\{
\begin{array}{ll}
x_j - \tau, & x_j > \tau \\
0, & |x_j| \leq \tau \\
x_j + \tau, & x_j < -\tau
\end{array} \right., \quad x \in \mathbb{E}, \quad \tau > 0, \quad j \inn{n}.
\end{gather}

\begin{table}[h]
\centering
\small
\caption{Oracle functions of the composite benchmark}
\label{label_154}
\begin{tabular}{lllll} \toprule
& $f(x)$ & $\Psi(x)$ & $f'(x)$ & $\operatorname{prox}_{\tau \Psi}(x)$ \\ \midrule
EN & $\frac{1}{2}\|A x - b \|_2^2$ & \multirow{2}{*}{$\lambda \|x\|_1 + \frac{\mu}{2} \|x\|_2^2$} & $A^T(Ax - b)$ & \multirow{2}{*}{$\frac{1}{1 + \tau \mu} \mathcal{T}_{\tau \lambda} (x)$} \\
ENLR & $\mathcal{I}(Ax) - \langle b, A x\rangle$ & & $A^T(\mathcal{L}(Ax) - b)$ & \\\bottomrule
\end{tabular}
\end{table}

In EN, $A$ is a square matrix ($M = n$) of size $n = 2500$ and $b$ a vector of size $n$. The entries of $A$, $b$ and the starting point $x_0$ are drawn independently from the standard normal distribution with scale $1$ for $A$ and $x_0$ and scale $5$ for $b$. The smoothness parameter of $f$ is given by $L_f = \sigma_{\max}^2(A)$, where $\sigma_{\max}(A)$ is the largest singular value of $A$. The problem parameters are the sparsifying coefficient $\lambda = 4$ and the strong convexity parameter $\mu = 10^{-4} L_f$. For both problems, all known strong convexity is contained in the regularizer $\Psi$ and we thus have $\mu_f = 0$ and $\mu_{\Psi} = \mu$.

In ENLR, the matrix $A$ is sparse and very large, of size $M = 50000$ by $n = 10000$, whereas the vector $b$ has $M$ values. All entries of $A$ are zero except $0.1\%$, at random locations, which are instead drawn from the standard normal distribution. The starting point $x_0$ has entries drawn from the standard normal distribution with scale $0.5$. The entries of $b$ are the random labels $b_i \in \{0, 1\}$, whose corresponding independent variables $B_i$ satisfy $\mathbb{P}(B_i = 1) = e^{(-A x_0)_i}$ for all $i \inn{M}$. The other problem parameters are $\lambda = 10^{-3}$, $L_f = \frac{1}{4} \sigma_{\max}^2(A)$, $\mu_f = 0$ and $\mu_{\Psi} = \mu = 10^{-4} L_f$.

We have tested 4 instances of EACGM, all with $\alpha_k = \alpha$, $k \geq 0$. In all instances, we have set $A_0 = 0$, $\gamma_0 = 1$ and $L_0 = L_f$ whereas the line-search parameters are $r_u = 2$ and $r_d = 0.9$ according to the original specification of ACGM. The termination thresholds for all methods are $\epsilon_{rel} = 10^{-5}$ for EN and $\epsilon_{rel} = 10^{-4}$ for ENLR. The accuracy threshold in ENLR is limited by the double precision floating point implementation of the objective smooth part.

The first instance of EACGM uses $\alpha = 0$, rendering it identical with original ACGM~\cite{ref_009}. The second employs the worst-case value of $\alpha = 0.7542$. The convergence analysis allows us to set in both situations $L_l = 0$. In the third instance of EACGM we consider $L_l = 0.1 L_f$. We have $q_l = 1/1001$ and set $\alpha = \alpha_l \overset{\operatorname{def}}{=} \alpha_{\operatorname{max}}(q_l) \approx 0.9780$. Lastly, we consider EACGM with $\alpha = 1$, a situation that violates our sufficient condition on convergence in \eqref{label_128}. The goal is to investigate whether the estimate sequence property in \eqref{label_018_comp} is still maintained by the algorithm. For simplicity, we set $L_l = 0$ in this case.

The convergence in iterate space of the four instances is shown for EN and ENLR in Figures \ref{label_155}(a) and \ref{label_155}(b), whereas the increase in the estimate sequence gaps, as given by \eqref{label_119} are shown for EN and ENLR in Figures \ref{label_155}(c) and \ref{label_155}(d), respectively.

\begin{figure}[h] \centering \footnotesize
\begin{minipage}[t]{0.48\linewidth} \centering
\includegraphics[width=\textwidth]{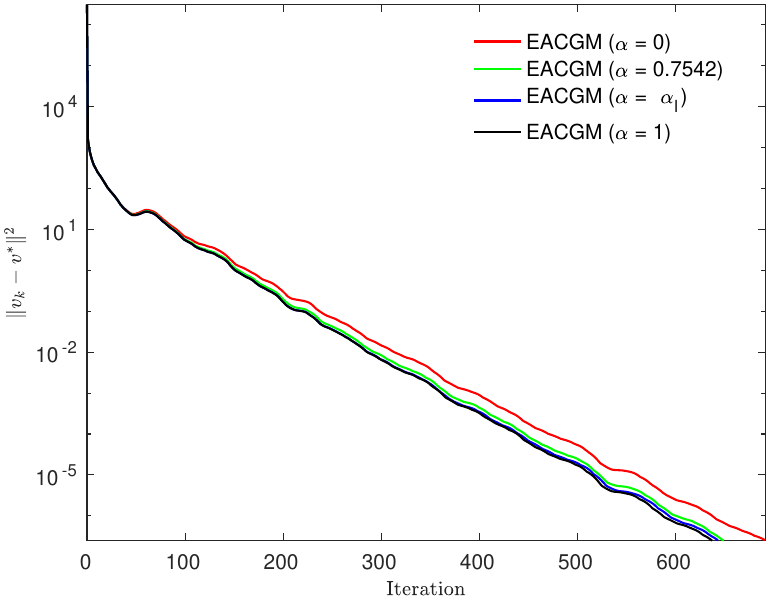}
(a) Convergence rates in iterate space on EN
\end{minipage}
\hspace{3mm}
\begin{minipage}[t]{0.48\linewidth} \centering
\includegraphics[width=\textwidth]{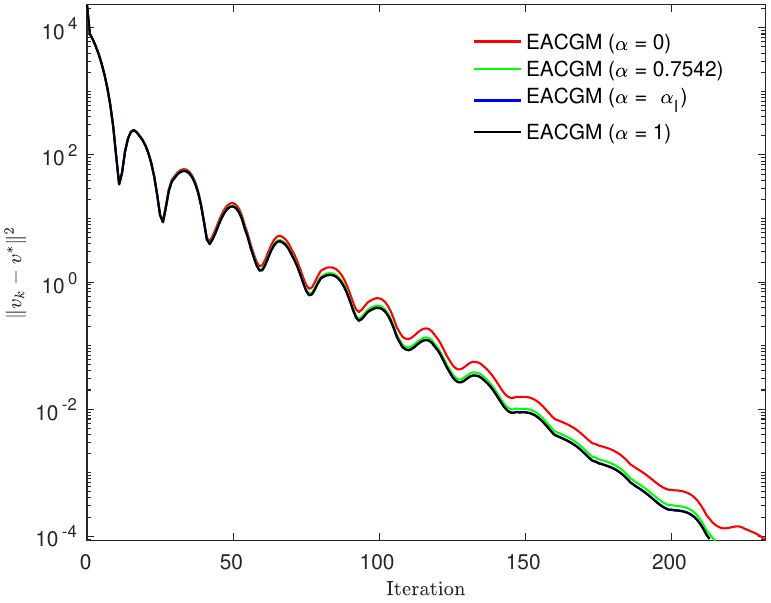}
(b) Convergence rates in iterate space on ENLR
\end{minipage}

\begin{minipage}[t]{0.48\linewidth} \centering
\includegraphics[width=\textwidth]{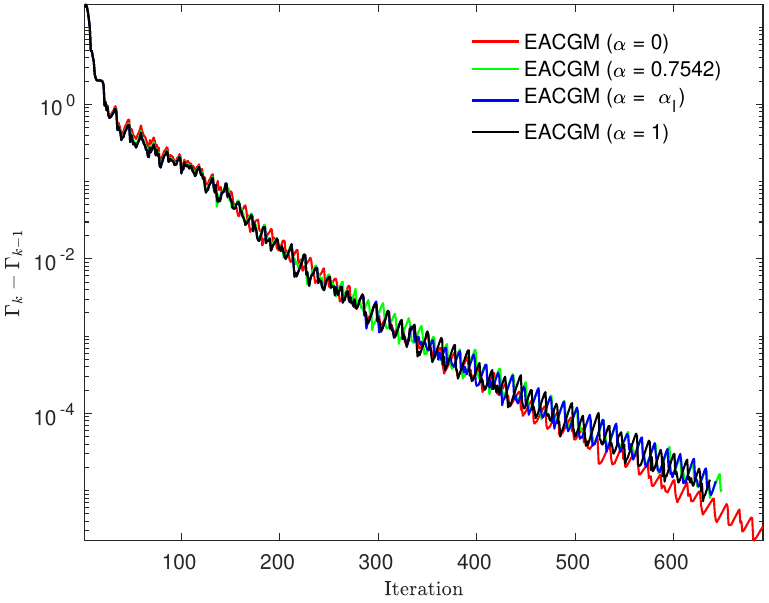}
(c) Estimate sequence gap increase on EN
\end{minipage}
\hspace{3mm}
\begin{minipage}[t]{0.48\linewidth} \centering
\includegraphics[width=\textwidth]{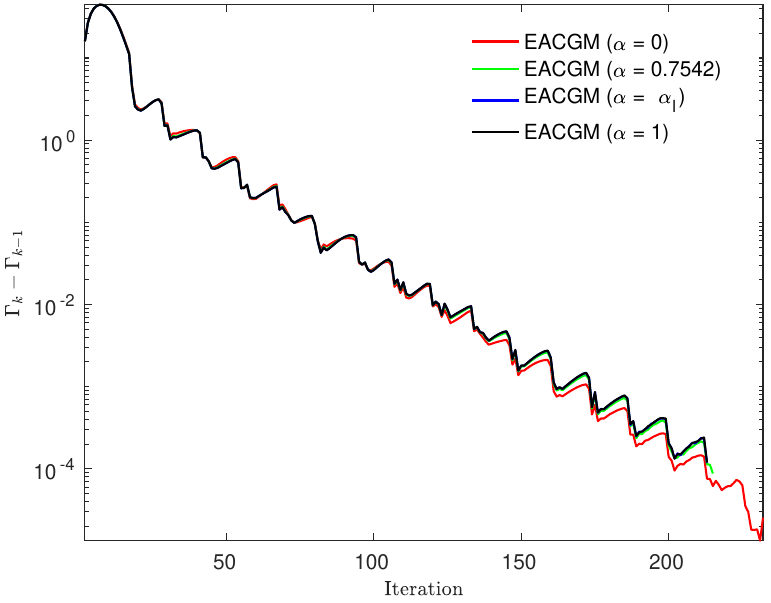}
(d) Estimate sequence gap increase on ENLR
\end{minipage}
\caption{Convergence behavior, measured both in $\|v_k - x^*\|^2$ vs. iterations as well as in the estimate sequence gap increase in \eqref{label_119}, of four instances of EACGM with $\alpha_k = \alpha$, $k \geq 0$, on the composite problem benchmark}
\label{label_155}
\end{figure}

The benefits of enhancement are less pronounced than in the case of memory, although there is a clear and consistent correlation between the value of $\alpha$ and performance: the higher the values of $\alpha$, the faster the convergence in iterate space. Nevertheless, the difference between worst-case $\alpha$ and $\alpha = 1$ is almost indistinguishable, with $\alpha_l$ lying in between. The discrepancy between the actual iterate convergence and the worst-case rate is however very large and suggests that real gains in performance are only to be seen under a very high accuracy regime that is beyond the capabilities of our implementation involving double precision floating point.

Surprisingly, the relationship between the estimate sequence gap increase and the value of $\alpha$ lacks the consistency previously seen in convergence behavior, with the increase being actually larger for $\alpha = 1$. In all instances, the increase is always positive implying that the estimate sequence property in \eqref{label_018_comp} is maintained by all parameter choices, even more so by $\alpha = 1$. This suggests that our analysis may be overly conservative with respect to the value of $\alpha$.

\section{Discussion}

The key theoretical development in this work consists in the introduction of a new mathematical object: the primal-dual estimate function described in \eqref{label_014}. Despite its simplicity, this object has remarkably powerful properties.

First of all, our estimate function is highly generative. Through the successive differentiation of the update in \eqref{label_024} followed by a worst-case analysis in Theorem~\ref{label_033} of a sufficient condition on maintenance of the estimate sequence property in \eqref{label_018}, we were able to derive a generalized Optimized Gradient Method, listed in Algorithm~\ref{label_047}, that encompasses as particular cases the original Optimized Gradient Method (OGM) in \cite{ref_002,ref_011}, the Information Theoretical Exact Method (ITEM) in \cite{ref_017,ref_001}, a strongly-convex generalization of OGM, and the related Triple Momentum Method (TMM) in \cite{ref_019}, as summarized in Table~\ref{label_055}.

By augmenting the estimate function we were able to recover the potential function used in \cite{ref_001} to analyze OGM, ITEM and TMM, only this time collectively while being able to assign a clear role to all constituent quantities. Specifically, in Lemma~\ref{label_057} we have that $\hat{w}_k$ is the known part of the lower bound at the optimum that we use to regularize our estimate function whereas $\gamma_k$ and $v_k$ are the estimate function curvature and primal optimum, respectively. The augmentation procedure is no longer primal as in \cite{ref_007}, but precisely follows the pattern introduced for analyzing the Accelerated Composite Gradient Method (ACGM) in \cite{ref_010,ref_003}. Thus the augmented estimate functions constitute a true relaxation of the estimate functions although this does not affect the overall guarantees. Unlike in \cite{ref_007}, where the potential function analysis yielded better guarantees than the estimate functions, here we see that identical results are obtained using both approaches.

By contrast with the uncomputable potential functions, the estimate functions can be maintained numerically in canonical form while the algorithm is running. This enables us to close the estimate sequence gap and thus dynamically increase the convergence guarantees at runtime. We have introduced the Optimized Gradient Method with Memory (OGMM) that enhances the benefits of this dynamic adjustment procedure by maintaining an aggregated subset of the oracle history. As an extension of our generalized OGM, OGMM also encompasses OGM, ITEM and TMM, having thus the highest known worst-case rate in the non-strongly convex case rate while having optimal and unimprovable guarantees in the strongly convex case, corresponding to a per iteration decrease of $(1 - \sqrt{q})^2$, where $q$ is the inverse condition number, in square distance to optimum or function residual. OGMM distinguishes itself from the previously introduced optimal methods through its adaptive mechanism. Preliminary simulations suggest that the convergence guarantee adjustment procedure can exploit both the local values of the Lipschitz constant as well as unknown strong convexity to a certain extent to produce a superior and remarkably stable convergence behavior.

At first glance, it appears that the mechanics our estimate function closely resemble those of the original estimate function used derive the Fast Gradient Method in \cite{ref_014}. This is indeed the case in the non-strongly convex case, when the additional term $\hat{w}_k$ disappears. The design pattern in \cite{ref_009,ref_010,ref_003} remains valid and we can confirm that the high worst-case rate of OGM is due to the increased tightness of the bounds, as extensively argued in \cite{ref_007}. However, the appearance of the $\hat{w}_k$ term in the strongly convex case underlines a key difference: whereas the role of the estimate function in FGM is to ensure a rapid decrease of the objective function value, the object introduced in this work instead focuses on decreasing the distance to the optimum in iterate space. The optimality of ITEM and TMM in function residual is merely a consequence of this aspect and does not describe the underlying mechanics.

The benefits of this observation are not confined to smooth unconstrained problems and can be readily applied to the much broader class of composite problems. We simply add the $\hat{w}_k$ term to the estimate functions used in deriving ACGM and directly obtain an Enhanced ACGM (EACGM) with improved guarantees in iterate space. Nevertheless, to guarantee convergence, we had to dampen the strongly convex terms in $\hat{w}_k$ with the parameters $\alpha_k$. We have provided two pessimistic strategies for updating the dampening parameters, one where the best case Lipschitz estimate of the smooth part is known and another that does not impose such constraints. We have shown that the worst-case guarantees in the pessimistic strategies are almost indistinguishable from the ideal $\alpha_k = 1$ case, as shown in \eqref{label_149} and Table~\ref{label_151}, an observation supported by preliminary computational results. Whether the ideal case is realizable over the entire composite problem class remains an open question.

As opposed to OGMM, which uses bounds that completely describe the smooth unconstrained problem class~\cite{ref_018}, the bounds employed in EACGM are not optimal, but are instead the tightest known that allow for the implementation of fully-adaptive line-search whereby the Lipschitz estimate can decrease at every iteration. The complete absence of slack terms in the analysis of OGMM argues that the corresponding bounds are used optimally. Under the ideal case of $\alpha_k = 1$ in EACGM, the analysis is also free of slack terms and suggests that the worst-case rate given by \eqref{label_149} and Table~\ref{label_151}, which approaches $1 - \sqrt{2 q}$ per iteration decrease in square distance to optimum for very small values of $q$, the case most commonly encountered in practice, may actually be the highest attainable under the fully-adaptive line-search requirement.

\section*{Acknowledgements}

We thank Yurii Nesterov for providing valuable comments and suggestions.

\end{document}